\documentclass{amsart}
\usepackage[utf8]{inputenc}
\usepackage{amsmath}
\usepackage{amssymb}
\usepackage{amsthm}
\usepackage{hyperref}
\usepackage{todonotes}
\usepackage[shortlabels]{enumitem}

\newcommand\cB{\mathcal{B}}
\newcommand\cC{\mathcal{C}}

\newcommand\cE{\mathcal{E}}

\newcommand\cM{\mathcal{M}}
\newcommand\cO{\mathcal{O}}

\newcommand\cU{\mathcal{U}}

\newcommand{\BB}{\mathcal B}
\newcommand{\EE}{\mathcal E}
\newcommand{\XX}{\mathcal X}

\newcommand{\N}{\mathbb{N}}
\newcommand{\R}{\mathbb{R}}
\newcommand{\Z}{\mathbb{Z}}

\newcommand{\C}{\mathbb{C}}
\newcommand{\CP}{\mathbb{CP}}

\newcommand{\rd}{{\rm d}}
\newcommand{\rT}{{\rm T}}
\newcommand{\rD}{{\rm D}}
\newcommand{\id}{{\rm id}}
\newcommand\pt{{\rm pt}}
\newcommand\pr{{\rm pr}}
\newcommand\ev{{\rm ev}}

\renewcommand{\ker}{{ \rm ker}}

\newcommand{\im}{{\rm im}}

\newcommand\eps{\varepsilon}
\renewcommand\epsilon{\varepsilon}
\renewcommand\phi{\varphi}
\renewcommand\tilde{\widetilde}
\renewcommand\bar{\overline}

\newcommand{\inner}[2]{\left\langle #1, #2\right\rangle}

\newtheorem{thm}{Theorem}[section]
\newtheorem{defn}[thm]{Definition}

\newtheorem{rmk}[thm]{Remark}
\newtheorem{lem}[thm]{Lemma}
\newtheorem{cor}[thm]{Corollary}

\newenvironment{itemlist}
   { \begin{list} {$\bullet$}
         { \setlength{\topsep}{.5ex}  \setlength{\itemsep}{0ex} \setlength{\leftmargin}{2.5ex} } }
   { \end{list} }

\title[A Polyfold proof of Gromov's Non-squeezing Theorem]{A Polyfold proof of\\
Gromov's Non-squeezing Theorem}
\author[Beckschulte, Datta, Seifert, Vocke, Wehrheim]{Franziska Beckschulte, Ipsita Datta, Irene Seifert,\\Anna-Maria Vocke, and Katrin Wehrheim}
\date{\today}

\begin{document}
\maketitle
\begin{abstract}
We re-prove Gromov's non-squeezing theorem by applying Polyfold Theory to a simple Gromov-Witten moduli space. 
Thus we demonstrate how to utilize the work of Hofer-Wysocki-Zehnder to give proofs involving moduli spaces of pseudoholomorphic curves that are relatively short and broadly accessible, while also fully detailed and rigorous. 
We moreover review the polyfold description of Gromov-Witten moduli spaces in the relevant case of spheres with minimal energy and one marked point.
\end{abstract}

\section{Introduction}
\label{sec:introduction}

The non-squeezing theorem, proven by Mikhail Gromov in 1985, essentially excludes nontrivial symplectic embeddings between balls $B_R$ and cylinders $Z_r$ of radius $R,r>0$ given by
\begin{align*}
B_R :&= B_R^n := \bigl\{ (x_i,y_i)_{i=1,\ldots,n} \in \R^{2n} \;\big|\; \textstyle \sum_{i=1}^n x_i^2 + y_i^2 \leq R \bigr\}, \text{ and} \\
Z_r :&= B^2_r \times \R^{2n-2} := \bigl\{ (x_i,y_i)_{i=1,\ldots,n} \in \R^{2n} \;\big|\; x_1^2 + y_1^2 \leq r \bigr\},
\end{align*}
in any dimension $2n\geq 4$.
More precisely, we equip both the closed balls and closed cylinders above with the standard symplectic form $\omega_{\text{st}} = \sum_{i=1}^n dx_i \wedge dy_i$ as subsets of $\R^{2n}$ with coordinates $x_1,y_1,\ldots,x_n,y_n$. Then it is easy to see that for any choice of $R$ and $r$ there are volume preserving embeddings $B_R \hookrightarrow Z_r$, due to the infinite length of the cylinder. If we only consider symplectic embeddings $\varphi:B_R \hookrightarrow Z_r$ with $\varphi^*\omega_{\text{st}}=\omega_{\text{st}}$, there are trivial embeddings for $R\leq r$. However, symplectic embeddings cannot exist for $R>r$, as was shown by Gromov \cite{Gromov} -- with various more detailed proofs published subsequently, e.g.\
\cite{hummel,HZ,McDuffSalamon2,Wendl}.\footnote{
Theorem~\ref{thm:nonsqueezing}, stated for closed balls and cylinders, implies the analogous statement for open balls and cylinders as follows: Assume there is a symplectic embedding $\varphi$ of the open ball of radius $R$ into the open cylinder of radius $r$, where $R>r$. Then there is $\varepsilon>0$ such that $R-\varepsilon>r + \varepsilon$, and $\varphi$ restricts to a symplectic embedding of the closed ball of radius $R-\varepsilon$ into the closed cylinder of radius $r+\varepsilon$. This is a contradiction to Theorem~\ref{thm:nonsqueezing}.
}

\begin{thm}
\label{thm:nonsqueezing}
If there is a symplectic embedding
$\varphi: B_R \hookrightarrow Z_r$,
then $R\leq r$.
\end{thm}

The idea of the proof is to construct an almost complex structure $J_1$ on the cylinder that pulls back to the standard complex structure $\varphi^*J_1 = J_{\textup{st}}$ on the ball, and find a non-constant $J_1$-holomorphic curve $C_1\subset Z_r$ passing through $\phi(0)$ with 
symplectic area at most $\pi r^2$. Then the pullback $\phi^{-1}(C_1)\subset B_R$ is a $J_{\textup{st}}$-holomorphic curve, and thus a minimal surface with respect to the standard metric on $\R^{2n}$. As it passes through the center of the ball $0\in B_R$, comparison with the flat disk of area $\pi R^2$ implies $R\leq r$ via a monotonicity lemma. 
To find such a $J_1$-holomorphic curve, one observes that the disk cross-section of $Z_r$ at the height of $\phi(0)$ has the required properties, except that it is holomorphic with respect to the standard complex structure $J_0$. The ideas is then to establish the existence of $J_t$-holomorphic curves for a path $J_t$ of almost complex structures connecting $J_0$ to $J_1$.
This requires subtle geometric analysis that is best performed by studying spheres in the closed symplectic manifold $\mathbb{CP}^1 \times T^{2n-2}$ as described in \S\ref{sec:compatifying-target-space}. In this setting, the main work is to find a pseudoholomorphic curve in the homology class $[\CP^1\times\{\pt\}]$ through a fixed point. 

\begin{thm}
\label{thm-intro:J1-sphere}
Given any point $p_0$ and compatible almost complex structure $J$ on $\mathbb{CP}^1 \times T$, there exists a $J$-holomorphic sphere $u : S^2 \rightarrow \mathbb{CP}^1 \times T$ with $p_0\in u(S^2)$ and homology class $u_*[S^2]=[\CP^1\times\{\pt\}]$.
\end{thm}

Assuming this existence, \S\ref{sec:using-monotonicity} proves Theorem~\ref{thm:nonsqueezing} by applying a monotoni\-city lemma for pseudoholomorphic maps that we discuss in Appendix~\ref{app:monotonicity-lemma-for-pseudoholomorphic-maps}.

\begin{rmk}\rm 
\label{rmk:J1-sphere}
We will prove Theorem~\ref{thm-intro:J1-sphere} for any compact symplectic manifold 
$(T,\omega_T)$ with ${\omega_T(\pi_2(T))=0}$, which excludes bubbling in the given homology class of minimal symplectic area. 
The torus satisfies this assumption since ${\pi_2(T)=0}$.
For more general symplectic manifolds, our line of argument still applies but the polyfold setup would require the inclusion of bubble trees. Moreover, this only proves the result with a possibly nodal $J$-curve. 
\end{rmk}

Classical non-squeezing proofs establish Theorem~\ref{thm-intro:J1-sphere} only for ``generic'' $J$ and require a delicate analysis of linearized Cauchy-Riemann operators to show that their surjectivity can be achieved alongside the condition $\varphi^*J = J_{\textup{st}}$.
Our proof of this more general result demonstrates what proofs of geometric statements look like when they can build on abstract polyfold theory \cite{HWZbook} and an existing polyfold description of the relevant moduli spaces, like for Gromov--Witten spaces given in \cite{HWZgw}. 
When analytical difficulties are outsourced to polyfold theory, the proof of Theorem~\ref{thm-intro:J1-sphere} becomes a transparent geometric argument:

\begin{itemize}
\item
The space $\cM(J)$ of solutions in Theorem~\ref{thm-intro:J1-sphere} modulo reparameterization is the zero set of a Fredholm section $\sigma_J: \cB \to \cE_J$ as described in \S\ref{sec:polyfold-setup}. 

\item
For the standard complex structure $J=J_0$ we show in Lemma~\ref{lem:only-one-J0-curve} that $\cM(J_0)=\{[u_0]\}$ consists of a unique solution $u_0(z)=(z, \pi_{T}(p_0))$. Theorem~\ref{thm:transverslity-at-0} moreover shows that the linearized operator ${\rm D}_{[u_0]}\sigma_{J_0}$ is surjective.

\item
Given any other $J=J_1$, \S\ref{sec:moduli-space} explains that a smooth family $(J_t)$ of compatible almost complex structures connecting it to $J_0$ gives rise to a compact family of moduli spaces $\bigsqcup_{t\in[0,1]}\cM(J_t)$. \S\ref{sec:polyfold-setup} identifies it with the zero set of the Fredholm section $\sigma(t,[u])=\sigma_{J_t}([u])$ over $[0,1]\times\cB$.

\item
If we assume $\cM(J_1)$ to be empty, this implies transversality of $\sigma$ over the boundary $\{0,1\}\times\cB$.
Then \S\ref{sec:applying-regularization-scheme} uses the polyfold regularization scheme explained in \S\ref{subsec:polyfold} to construct a compact $1$-dimensional cobordism $(\sigma+p)^{-1}(0)$ between $\cM(J_0)=\{[u_0]\}$ and $\cM(J_1)=\emptyset$. This proves Theorem~\ref{thm-intro:J1-sphere} by contradiction. 
\end{itemize}

\begin{rmk} \rm
The exact meaning of `Fredholm section' in the first step determines the class in which the last step provides a contradiction. 
Minimal work in the first step would be to cite \cite{ben,HWZgw} for a general description in which $\cB$ is a polyfold (possibly containing nodal curves or curves with nontrivial isotropy). However, this would force us to work with multivalued perturbations $p$ and discuss weighted branched orbifolds in the last step.\footnote{The perturbation $p$ is generally a `multisection functor' resulting in $\tilde\cM^p:=(\sigma+p)^{-1}(0)$ being a weighted branched orbifold. The contradiction to its boundary (in appropriate orientation) $\partial\tilde\cM^p=\cM(J_0)$ being a single point (of trivial isotropy and weight $1$) then arises from Stokes' Theorem 
$0 = \int_{\tilde\cM^p} \rd (1) = \int_{\partial \tilde\cM^p} 1 = 1$ which holds in this context by \cite[\S9.5]{HWZbook}.}
Instead, we show in \S\ref{sec:polyfold-setup} that the specific choice of homology class rules out isotropy, so that we are working with an M-polyfold $\cB$. Then $p$ is single-valued, and the perturbed solution set $(\sigma+p)^{-1}(0)$ is a manifold, contradicting the fact that compact $1$-manifolds have an even number of boundary points. 
\end{rmk}

The above outline uses polyfold theory entirely as a ``black box'' with two features: (a) it describes compactified moduli spaces as zero sets of `Fredholm sections'; (b) such `Fredholm sections' can be perturbed to regularize the moduli space.
To demystify feature (a), \S\ref{sec:polyfold-setup} gives an introduction to the rather technical polyfold description of general Gromov-Witten moduli spaces by going through the details of \cite{HWZgw} for our specific case. Since we consider curves of minimal positive symplectic area, we can exclude nodal curves arising from bubbling so that our ambient space $\cB$ is the space of maps $S^2\to \CP^1\times T$ (of Sobolev class $W^{3,2}$) modulo M\"obius transformations of $S^2$ that fix a marked point. Moreover, we show that this space has trivial isotropy, that is, none of these maps -- holomorphic or otherwise -- is invariant under reparameterization with a nontrivial M\"obius transformation. 
This gives the ambient space $\cB$ the structure of a sc-Hilbert manifold. 
We describe this notion in the following section as part of a brief introduction to polyfold theory. 
This section also demystifies feature (b) by stating the perturbation theorem in the trivial isotropy case that is relevant for the non-squeezing proof.

\smallskip\noindent
\subsubsection*{Acknowledgements.}
This project was initiated by the `Women in Symplectic and Contact Geometry and Topology' workshop \href{https://icerm.brown.edu/topical_workshops/tw19-4-wiscon}{(WiSCon)} at ICERM in 2019. We are deeply grateful to all the organizers and fellow participants for giving us the space and inspiration for this work! Further gratitude to Ben Filippenko and Wolfgang Schmaltz for helpful discussions.\\
This work is supported by Deutsche Forschungsgemeinschaft (DFG) under Germany’s Excellence Strategy EXC-2181/1 - 390900948 (the Heidelberg STRUCTURES Excellence Cluster), CRC/TRR 191 and RTG 2229, and by the US National Science Foundation grants DMS-1708916 and DMS-1807270.

\subsection{Polyfold notions and regularization theorems}
\label{subsec:polyfold}

Polyfold theory was developed by Hofer, Wysocki, and Zehnder (see \cite{HWZbook} and the citations therein) as a general solution to the challenge of regularizing moduli spaces of pseudoholomorphic curves. 
The expectation is that any compact moduli space $\bar\cM$ that is described as zero set of a section can be regularized by appropriate perturbations of the section. For smooth sections in finite dimensions this is proven in e.g.\ \cite[ch.2]{GuillP}.

\begin{thm}[Finite dimensional regularization] \label{thm:findim}
	Let $E \rightarrow B$ be a smooth finite dimensional vector bundle, and let $s: B\rightarrow  E$ be a smooth section such that $s^{-1}(0)$ is compact. Then there exist arbitrarily small, compactly supported, smooth perturbation sections $p:B\rightarrow E$ such that $s+p$ is transverse to the zero section, and hence $(s+p)^{-1}(0)$ is a manifold of dimension $\textup{dim}(B) - \textup{rank}(E)$. 
	
	Moreover, the perturbed zero sets $(s+p')^{-1}(0)$ and $(s+p)^{-1}(0)$ of any two such perturbations $p,p':B\rightarrow E$ are cobordant.
\end{thm}

A direct generalization of this theorem applies when $B$ has boundary giving rise to perturbed zero sets with boundary $\partial (s+p)^{-1}(0) = (s+p)^{-1}(0) \cap \partial B$. 
Unfortunately, moduli spaces of pseudoholomorphic curves generally do not have natural descriptions to which this theorem applies. There are several reasons: 

\begin{enumerate}[(1)]
    \item 
    The space of pseudoholomorphic maps $u:S\to M$, with fixed domain $S$ and target $M$, is the zero set of a Fredholm section of a Banach bundle. But due to `bubbling and breaking' phenomena, its compactification contains maps defined on different domains\footnote{Some `SFT neck stretching' moduli spaces also vary the target space $M$.}. To put a topology on the resulting set of maps from varying domains, one uses `pregluing constructions' to transfer maps from `nodal or broken' domains to nearby smooth domains. However, these constructions do not provide homeomorphisms to open subsets of Banach spaces (see e.g.\ \cite[p.10]{usersguide}), so they do not yield local charts for any classical generalization of (topological or smooth) manifolds.
    \item 
    The moduli spaces typically arise as quotients of spaces of pseudoholomorphic maps by groups of repa\-ra\-meterizations of their domains. If these groups act with nontrivial isotropy, then we expect an orbifold structure on any ambient space that contains the moduli space. 
    \item
    If we wish for a differentiable structure on this ambient space, then we need to ensure that repa\-ra\-meterizations act differentiably. However, this is not the case for the classical Banach manifold structures on spaces of maps. As a result, while there are local charts for maps-modulo-reparameterization with fixed domain (constructed as local slices to the group action), the transition maps between different charts are nowhere differentiable.
 \end{enumerate}  
    
The classical regularization constructions for moduli spaces of pseudoholomorphic curves, such as \cite{McDuffSalamon2}, work around these problems by finding geometric perturbations that achieve transversality for spaces of maps with fixed domain. Once they achieve finite dimensional spaces of perturbed solutions, this requires further steps to take quotients and compactify. Whether or not there are sufficiently many geometric perturbations that are both equivariant and compatible with gluing constructions depends on the particular geometric setting. The classical proof of the non-squeezing theorem makes use of the geometric setting of `least energy' to rule out (1) nodal curves as well as (2) isotropy (due to multiple covers), so that only (3) the differentiability challenge is present. The latter is resolved by finding a regular choice of $J$ with $\phi^*J=J_{\text{st}}$. This requires showing that holomorphic maps in the desired homology class all pass injectively through a part of $(\CP^1\times T) \setminus \phi(B_R)$, where we can freely vary the almost complex structure.
While this approach yields a rigorous proof, it requires -- even in the simplest case -- sophisticated analysis combined with specific geometric properties of the holomorphic curves.

Polyfold theory, on the other hand, uses the above challenges as a guide to generalize the notion of a section in Theorem~\ref{thm:findim} so that the abstract perturbation theory applies to the desired moduli spaces, and no further case-specific analysis or geometric properties are needed. The main features are as follows:

\begin{enumerate}[(1)]
    \item 
    The `pregluing constructions' generalize open subsets of Banach spaces to images of retractions as new local models. For example, the neighbourhood of a nodal sphere is described by an open subset of the model space $\im\,\rho=\bigcup_{a\in\C} \{a\}\times \pi_a(V)$ that arises from a family of projections $\rho:\C\times V\to \C\times V, (a,v)\mapsto (a,\pi_a v)$ on a Banach space $V$, which are
    centered at $\pi_0=\id$ and $\im\,\pi_a \subsetneq V$ for $a\neq 0$.
    For further details see e.g.\ \cite[2.3]{usersguide}. 
    \item 
    The orbifold structure is captured by formulating the notion of an atlas as a groupoid. This provides a nonsingular structure as in (1) on the object space (where e.g.\ perturbations are constructed), with isotropy appearing in the morphisms (e.g.\ forcing the perturbations to become multivalued). 
    \item
    Differentiability of transition maps between different local charts (resp.\ the structure maps in the groupoid) is achieved by defining a new notion of scale-differen\-tiability for maps between Banach spaces equipped with an additional scale structure. These notions are obtained by formalizing the differentiability features of reparameterization maps between Sobolev spaces into a notion that satisfies a chain rule
    (see \cite[2.2]{usersguide}).
 \end{enumerate}  

Restricted to finite dimensional Banach spaces, the retractions in (1) are trivial, and (3) coincides with classical differentiability, so (2) reproduces the notion of an orbifold being represented by a proper \'etale groupoid 
(see e.g.\ \cite{Moe} for an introduction). 
In infinite dimensions, these generalizations yield the following new notions. Here and throughout, we restrict the notions of \cite{HWZbook} to metrizable topologies\footnote{Metrizability of polyfolds is guaranteed by paracompactness and \cite[Thm.7.2]{HWZbook}.} and sc-Hilbert spaces\footnote{Hilbert spaces equipped with scale structures automatically admit scale-smooth bump functions by \cite[\S5.5]{HWZbook}. These are crucial for the existence of transverse perturbations.}. 
In each case, sc-compatibility means that the transition maps are scale-smooth.

\begin{itemize}
    \item A \textbf{sc-Hilbert manifold} is a metric space equipped with sc-compatible local homeomorphisms to open subsets of sc-Hilbert spaces.
    \item An \textbf{M-polyfold} is a metric space equipped with sc-compatible local homeomorphisms to open subsets of scale-smooth retracts in sc-Hilbert spaces. 
    \item A \textbf{polyfold} is a metric space equipped with sc-compatible local homeomorphisms to finite quotients of open subsets of scale-smooth retracts in sc-Hilbert spaces. These domains with group actions and lifts of transition maps form a proper groupoid whose object and morphism spaces are M-polyfolds, and whose structure maps are local sc-diffeomorphisms.
    \item A sc-Hilbert manifold/M-polyfold/polyfold $\cB$ \textbf{with boundary}\footnote{For an introduction to the notion of corners in polyfold theory see \cite[\S5.3]{usersguide}. Note, however, that corners usually appear together with coherence conditions -- in which perturbations on boundary strata need to coincide with the perturbations of other moduli spaces which are identified with these boundary strata. The construction of coherent perturbations then requires not just a workable notion of boundary and corner strata, but an ordering of the moduli spaces that prevents circular coherence conditions when constructing perturbations.  
    } 
    is a space with compatible charts as before but allowing for open subsets of $[0,\infty)\times H$, where $H$ is a sc-Hilbert space. Its boundary $\partial\cB$ is the union of all preimages of $\{0\}\times H$. 
\end{itemize}

\begin{rmk} \rm 
\label{rmk:finite-dim-mfd-sc-Hilbert-mfd}
Every manifold $M$ (with boundary) is a sc-Hilbert manifold: It is locally homeomorphic to open subsets of $\R^n$ (or $[0,\infty)\times\R^{n-1}$). Here each $\R^k$ is a sc-Hilbert space with trivial scale structure; see \cite[Ex.4.1.8]{usersguide}.
\end{rmk}

With this language in place, the application of polyfold theory to a given moduli space $\bar\cM$ -- for example the moduli space in Theorem~\ref{thm-intro:J1-sphere} -- has two steps:  
\begin{enumerate}[(a)]
\item
Describe $\bar\cM\cong\sigma^{-1}(0)$ as the zero set of a section $\sigma:\cB\to\cE$ over a polyfold or M-polyfold $\cB$ with $\cE$ a `strong bundle' and $\sigma$ `scale-Fredholm' as defined in \cite{usersguide,HWZbook}.
Intuitively, $\cB$ is the same space of possibly-nodal-maps-modulo-reparameterization as $\bar\cM$ but allows for general maps in some Sobolev space, with the pseudoholomorphic condition encoded in the section $\sigma([u])=[\bar\partial_J u]$ of an appropriate bundle $\cE$.
This step is best achieved by combining existing polyfold descriptions such as \cite{HWZgw} with general construction principles such as restrictions \cite{ben}, pullbacks \cite{wolfgang}, quotients \cite{zhengyi}.
For our example, we describe this in detail in \S\ref{sec:polyfold-setup}. 

\item
Apply the corresponding regularization theorem to draw the desired conclusions. For general polyfolds (with boundary) this is \cite[Thm.15.4 (15.5)]{HWZbook} and involves multivalued perturbations. In our example, the M-polyfold versions in Theorem~\ref{thm:polyfold-regularization-scheme} and Remark~\ref{rmk:perturbation-supported-inside-interval} below suffice. 
\end{enumerate}

The following generalization of the finite dimensional regularization Theorem~\ref{thm:findim} is proven in \cite[Theorems 3.4, 5.5, 5.6]{HWZbook}. It resolves the challenges (1) and (3) above, hence covering the case required for the non-squeezing proof. 

\begin{thm}[M-polyfold regularization]
\label{thm:polyfold-regularization-scheme}
	Let $\EE\rightarrow\BB$ be a strong M-polyfold bundle, and let $\sigma:\BB\rightarrow\EE$ be a scale-smooth Fredholm section such that $\sigma^{-1}(0)$ is compact. Then there exists a class of perturbation sections $p:\BB\rightarrow\EE$ supported near $\sigma^{-1}(0)$ such that $(\sigma+p)^{-1}(0)$ carries the structure of a smooth compact manifold of dimension $\textup{index}(\sigma)$ with boundary $\partial (\sigma+p)^{-1}(0)= (\sigma+p)^{-1}(0) \cap \partial\BB$. 
	
	Moreover, for any other such perturbation $p':\BB\rightarrow\EE$, there exists a smooth cobordism between $(\sigma+p')^{-1}(0)$ and $(\sigma+p)^{-1}(0)$.
\end{thm}

\begin{rmk} \rm
\label{rmk:perturbation-supported-inside-interval}
Suppose that the section $\sigma$ in Theorem~\ref{thm:polyfold-regularization-scheme} restricts to a transverse section on the boundary, i.e.\ $\sigma|_{\partial\BB}: \partial\BB \to \EE|_{\partial\BB}$ has surjective linearizations at all points in $\sigma^{-1}(0)\cap \partial\BB$. 
Then we can choose the perturbation section $p$ to be supported in the interior, i.e.\ $p|_{\partial\BB}\equiv 0$. As a result, $(\sigma+p)^{-1}(0)$ has boundary $\partial (\sigma+p)^{-1}(0)= (\sigma+p)^{-1}(0) \cap \partial\BB = (\sigma|_{\partial \BB})^{-1}(0)$. 

This can be proven by following the proof of the regularization theorem in \cite{HWZbook}. It is explicitly stated and proven in the last part of \cite[Thm.A9]{pss}. In our case, the map $e:\XX\to\emptyset$ and submanifolds $C_i=\emptyset$ are trivial, and the polyfold $\XX=\BB$ has trivial isotropy. So, the `multisection' $\lambda$ will be represented by a perturbation section $p:\cB\to\cE$.
Our transversality assumption on the boundary means that the `trivial multisection $\lambda^\delta$ representing' $p|_{\partial\BB}\equiv 0$ yields an `admissible ... multisection in general position to the perturbed zero set in the boundary' $\{x\in\partial\BB \,|\, \sigma(x) = 0\}$. The conclusion is the existence of a perturbation section $p$ with $p|_{\partial\BB}\equiv 0$ so that $\sigma+p$ is `admissible' and in `general position', as required for the conclusions of Theorem~\ref{thm:polyfold-regularization-scheme}.
\end{rmk}

For readers interested in regularization theorems that resolve the challenge (2) of nontrivial isotropy, we recommend the brief overview \cite[Rmk.2.1.7]{usersguide} and the in-depth discussion of the finite dimensional case \cite{dusa-groupoids} before diving into the technicalities of \cite{HWZbook} or their summary in \cite{wolfgang}. 
Despite a lot of notational overhead, the general polyfold regularization theorem \cite[Thm.15.4]{HWZbook} can be understood as a direct combination of the regularization theorems for sections over M-polyfolds and finite dimensional orbifolds.

\section{Outline of the Proof}
\label{sec:outline-of-the-proof}

Let us consider a symplectic embedding $\varphi: B_R \hookrightarrow Z_r$ for radii $R>0$ and $r>0$. We will prove $R\leq r$ by showing that $R'\leq r+\varepsilon$ for any choice of $0<R'<R$ and $\varepsilon>0$. These choices are needed for constructions in the following section.

\subsection{Compactifying the target space}
\label{sec:compatifying-target-space}

The proof uses the theory of pseudoholomorphic curves. Since the analytic setup is simpler for closed manifolds, we prefer to work with a compact target space. For that purpose we fix an $\varepsilon >0$. Then we can understand $\varphi$ as an embedding
$$\varphi: B_R \hookrightarrow \mathring{Z}_{r+\varepsilon}$$
into the slightly larger open cylinder $\mathring{Z}_{r+\varepsilon} = \mathring{B}^2_{r+\varepsilon} \times \R^{2n-2}$.
The first factor of this cylinder compactifies to a $\mathbb{CP}^1$. The standard symplectic form on $\mathring{B}^2_{r+\varepsilon}$ descends to a symplectic (and thus, area) form $\omega_{\mathbb{CP}^1}$ such that $\mathbb{CP}^1$ has area
\begin{align*}
   \int_{\mathbb{CP}^1} \omega_{\mathbb{CP}^1} = \pi(r+\varepsilon)^2.
\end{align*}
So we may view $\varphi$ as a symplectic embedding 
\begin{align*}
    \varphi \colon ( B_R , \omega_{\text{st}} ) \hookrightarrow ( \mathbb{CP}^1 \times \mathbb{R}^{2n-2} , \omega_{\mathbb{CP}^1} \oplus \omega_{\text{st}} ).
\end{align*}
Now, we want to compactify the second factor of the cylinder as well. Remember that $B_R$ is the closed ball. So, the projection to $\R^{2n-2}$ of its image under the continuous map $\varphi$ is compact.
This means that we can choose $N > 0$ sufficiently large such that 
$\varphi ( B_{R} ) \subset \mathbb{CP}^1 \times (-\frac 12 N , \frac 12 N )^{2n-2}$.
Then we can view $\varphi(B_R)$ as a subset of the $(2n-2)$-dimensional torus $T:= \R^{2n-2}/ N\Z^{2n-2}$ with standard symplectic form $\omega_T$ induced from $\omega_{\text{st}}$ on $\R^{2n-2}$. This means we get a symlplectic embedding (again denoted by) $\varphi$,
$$ 
\varphi: (B_{R}, \omega_{\text{st}}) \hookrightarrow (\mathbb{CP}^1 \times T, \omega:= \omega_{\mathbb{CP}^1} \oplus \omega_{T} ).
$$
The proof now proceeds by studying pseudoholomorphic curves in $\CP^1\times T$.
Here we wish to work with an almost complex structure $J_1$ on $\mathbb{CP}^1 \times T$ so that the pullback of $J_1$-holomorphic maps under the embedding $\phi$ yields pseudoholomorphic maps to $B_R$ with respect to the standard complex structure $J_{\textup{st}}$ on $B_R\subset\R^{2n}$. 
This is crucial for the last step of the non-squeezing proof in \S\ref{sec:using-monotonicity} which uses monotonicity with respect to the standard metric $\omega_{\text{st}}(\cdot,J_{\textup{st}} \cdot)$ on $B_R$. 
To do this rigorously, we need to shrink the ball slightly to interpolate between almost complex structures. 

\begin{lem}  \label{lem:J1}
For any $0<R'<R$ there is an almost complex structure $J_1$ on $\mathbb{CP}^1 \times T$ that is compatible with $\omega= \omega_{\mathbb{CP}^1} \oplus \omega_{T}$ and satisfies $\varphi^*J_1|_{B_{R'}} = J_{\textup{st}}$.
\end{lem}
\begin{proof}
The basic idea is to define $J_1 = \varphi_*J_{\textup{st}}$ on the image of $\varphi$ and to set $J_1 = J_0$ outside a neighbourhood of the image. But as a weighted sum of two almost complex structures will in general not be an almost complex structure, we can not directly interpolate between these.  Instead, we interpolate between the corresponding Riemannian metrics $g_0:=\omega (\cdot,J_0\cdot)$ on $\CP^1\times T$ and $g_{\varphi_*J_{\textup{st}}}:= \omega (\cdot,\varphi_*J_{\textup{st}} \cdot)$ on $\varphi(B_R)$. To do this, we choose a partition of unity $\psi_0+\psi_1=1$ subordinate to the cover $\CP^1\times T = U_0 \cup U_1$ where $U_0:= \CP^1\times T \setminus \varphi(B_{R'})$ and $U_1:=\varphi(\mathring{B}_R)$. (These are open subsets because $\varphi$, being an embedding, maps open/closed subsets to open/closed subsets.) Since $\psi_i$ is supported in $U_i$ and $\varphi(B_R')\cap U_0 = \emptyset$, we have $\psi_1|_{\varphi(B_R')}\equiv 1$ and thus obtain a metric $g_1$ with $g_1|_{\varphi(B_R')}=g_{\varphi_*J_{\textup{st}}}$ by interpolating with this partition of unity,
$$
g_1 \,:=\; \psi_0 \cdot g_0 \;+\; \psi_1\cdot g_{\varphi_*J_{\textup{st}}}. 
$$
Finally, a pair of a Riemannian metric $g$ and a symplectic form $\omega$ determine an almost complex structure $J$ compatible with $\omega$ and if $g$ was of the form $g=\omega(\cdot,J\cdot)$, then the determined almost complex structure is in fact the same $J$, see  \cite[Prop. 2.50 (ii)]{McDuffSalamon1}.
Thus, $g_1$ and $\omega$ together determine an almost complex structure $J_1$ that has the required properties.
\end{proof}

There are two more properties of pullbacks $C_0:=\phi^{-1}(C_1)\subset B_{R'}$ of $J_1$-ho\-lo\-mor\-phic curves $C_1 \subset \CP^1\times T$ that are required for their to prove the non-squeezing result $R'\leq r+\eps$. First, we need $C_1$ to pass through the point $p_0:=\varphi(0)\in \mathbb{CP}^1 \times T$, so that $C_0\subset B_{R'}$ passes through the center $0$ of the ball. Second, we wish to bound the symplectic area $\int_{C_0} \omega_{\text{st}} \leq \int_{C_1} \omega \leq \pi (r+\eps)^2$. The latter is achieved by prescribing the homology class $[C_1]=[\CP^1\times\{\pt\}]$ since this determines the integral of the closed symplectic form $\omega$, 
$$
\int_{C_1} \omega \;=\; \int_{\CP^1\times\{\pt\}} \omega_{\mathbb{CP}^1} \oplus \omega_{T}
\;=\; \int_{\CP^1} \omega_{\mathbb{CP}^1}  +  \int_{\{\pt\}} \omega_{T} \;=\; \pi (r+\eps)^2 + 0. 
$$
Ultimately, we will find a not necessarily embedded curve $C_1=u(S^2)$ by studying $J_1$-holomorphic maps $u:(S^2,i) \rightarrow (\mathbb{CP}^1 \times T^{2n-2}, J_1)$ with a point constraint $u(z_0)=p_0$ in the homology class $[u]=[\CP^1\times\{\pt\}]$.
Their existence is stated, for general $J$ on a product manifold $\mathbb{CP}^1 \times T$, 
in Theorem~\ref{thm-intro:J1-sphere}. The proof starts with the existence of a unique $J_0$-holomorphic map for a specific $J_0$ described in \S\ref{sec:unique-J0-curve}, and is completed in \S\ref{sec:applying-regularization-scheme} based on the polyfold constructions in \S\ref{sec:polyfold-setup}.

\subsection[The unique J0-holomorphic curve]{The unique \texorpdfstring{$J_0$}{J0}-holomorphic curve}
\label{sec:unique-J0-curve}

This section begins our study of pseudoholomorphic curves in $\CP^1\times T$ by considering a split almost complex structure $J_0= i \oplus J_T$ on $\CP^1 \times T$.
Here $(T,\omega_T)$ can be any compact symplectic manifold, though its topology will be restriced in following sections. 
For the nonsqueezing proof, there is a standard complex structure $J_T$ on the torus $T= \R^{2n-2}/ N\Z^{2n-2}$; in general, we choose any $\omega_T$-compatible almost complex structure $J_T$ on $T$. This ensures that $J_0$ is compatible with $\omega=\omega_{\mathbb{CP}^1} \oplus \omega_{T}$, meaning $g_0=\omega (\cdot, J_0\cdot)$ is a Riemannian metric. 
In the following, we view $S^2$ as a Riemann surface by identifying it with $\CP^1$ and using the standard complex structure $i$ on $\CP^1$. Then we find a $J_0$-holomorphic sphere passing through any given point $p_0=(z_0,m_0)\in \mathbb{CP}^1 \times T$ by combining the identification $S^2\cong\CP^1$ with a constant map to $T$,
$$ 
u_0 : (S^2,i) \rightarrow (\mathbb{CP}^1 \times T, J_0) , \quad z \mapsto (z, m_0 )  .
$$
The symplectic area of this sphere -- a quantity that only depends on the homology class -- is
\begin{align}
    \label{eq:area-of-u0}
    E(u_0) = \int_{S^2} u_0^*\omega = \int_{\CP^1} \omega_{\mathbb{CP}^1} = \pi (r+\varepsilon)^2 .
\end{align}
The next lemma shows that, up to reparameterization, this is the only sphere with these properties in its homology class.

\begin{lem}
\label{lem:only-one-J0-curve}
Assume $u: (S^2,i) \rightarrow (\mathbb{CP}^1 \times T, J_0)$ is $J_0$-holomorphic, passes through $p_0$, and represents the class $[\mathbb{CP}^1 \times \{\textup{pt}\}]\in H_2(\mathbb{CP}^1 \times T;\mathbb{Z})$. Then, there is a biholomorphism $\psi: (S^2,i) \rightarrow (S^2,i)$ such that $u \circ \psi = u_0$.
\begin{proof}
Since $u$ is $(J_0=i\oplus J_T)$-holomorphic, its composition with projection to each factor yields holomorphic maps $f:= \text{pr}_{\mathbb{CP}^1} \circ u : S^2 \rightarrow \mathbb{CP}^1 \cong S^2$ and $g:=\text{pr}_{T} \circ u : S^2 \rightarrow T$.
Moreover, the homology condition specifies $g_*[S^2]=[\{\textup{pt}\}]=0\in H_2(T)$ and $f_*[S^2]=[S^2]\in H_2(S^2)$. 
The energy identity $\int g^*\omega = \int \frac 12 |\rd g|^2$ (see \cite[Lemma~2.2.1]{McDuffSalamon2}) then implies $\int |\rd g|^2 =0$. So, $g$ must be constant. Since $u$ passes through $p_0=(z_0,m_0)$, this means $g(z) = \text{pr}_{T} ( u(z) ) =m_0$. 
Moreover, $f_*[S^2]=[S^2]\in H_2(S^2)$ implies that $f$ is neither constant nor a multiple cover of another holomorphic map. Thus, $\psi := f^{-1}$ exists and is a biholomorphism of $S^2$. 
Then we obtain the claim as
$( u\circ \psi )(z) = \bigl(f (f^{-1}(z)) \,,\, g(\psi(z) \bigr)  = (z , m_0 ) = u_0$ for all $z\in S^2$. 
\end{proof}
\end{lem}

\begin{rmk} \rm 
It is part of both the classical and our polyfold proof to show that the curve $u_0$ is transversely cut out of the space of all curves in its homology class passing through $p_0$.
This statement will be made precise in Theorem~\ref{thm:transverslity-at-0}.
\end{rmk}

\subsection{Using the monotonicity lemma}
\label{sec:using-monotonicity}

This section finishes the proof of the nonsqueezing Theorem~\ref{thm:nonsqueezing} assuming that we have found a $J_1$-holomorphic map $u_1:S^2\to \CP^1\times T$ with the same properties as the unique $J_0$-holomorphic curve in Lemma~\ref{lem:only-one-J0-curve}, except that $J_1$ is an almost complex structure as in Lemma~\ref{lem:J1} with $\phi^*J_1=J_{\text{st}}$. The existence of $u_1$ follows from Theorem~\ref{thm-intro:J1-sphere}, proven in \S\ref{sec:applying-regularization-scheme}. 
Given such $u_1:S^2\to \CP^1\times T$, we obtain a $J_{\text{st}}$-holomorphic map
on $\tilde S:=u_1^{-1}(\varphi(\mathring{B}_{R'}))\subset S^2$, 
$$
v:=\varphi^{-1}\circ u_1 :  \; \tilde S \longrightarrow \R^{2n}. 
$$
This map passes through the center $\varphi^{-1}(p_0)=0$ of the ball $B_{R'}$ and has area $\int v^*\omega_{\text{st} }\leq \int u_1^*\omega =\pi (r+\varepsilon)^2$, so comparison with the minimal surface through the center of the ball -- the disk of area $\pi (R')^2$ -- will yield $R' \leq r+\eps$. 
To deduce this inequality directly from the monotonicity lemma for minimal surfaces, we would have to show that $v$ is an embedding. Instead, we use a monotonicity lemma for holomorphic maps to a complex Hilbert space. It can be found as Lemma~\ref{lem:holmonotone} in the appendix, together with a detailed proof.

\begin{proof}[Proof of Theorem~\ref{thm:nonsqueezing}] 
We will apply Lemma~\ref{lem:holmonotone} to the map $v:\tilde S \to \R^{2n}$ with $(V,J):=(\R^{2n},J_{\text{st}})$ and open balls $\mathring{B}_{R_k}:=\mathring{B}_{R_k}(0)\subset\R^{2n}$ of radii $R_k\to R'$. 
The preimage of the center is nonempty, $v^{-1}(\{0\})=u^{-1}(\{p_0\})\neq \emptyset$, since $u_1$ passes through $p_0=\phi(0)$. 
The domain $\tilde S$ of $v$ is an open subset of $S^2$ because $\varphi(\mathring{B}_{R'})\subset\CP^1\times T$ is the image of an open set under an embedding.
To apply the lemma we need to restrict $v$ to a compact subdomain $S_k\subset\tilde S$ with smooth boundary such that $\|v(z)\|\geq R_k$ for all $z\in\partial S_k$. For that purpose we consider the smooth function $\rho:\tilde S \to \R, z \mapsto \|v(z)\|^2$. Since its regular values are dense we can find a sequence $0<R_k<R'$ with limit $\lim_{k\to\infty}R_k= R'$ such that $R_k^2$ are regular values of $\rho$. Then $S_k:= \{z\in \tilde S \,|\, \|v(z)\|\leq R_k\}$ is a domain with smooth boundary $\partial S_k = \rho^{-1}(R_k^2)$. It is compact because $S_k=v^
{-1}(B_{R_k})=u_1^{-1}(\phi(B_{R_k})\subset S^2$ is a closed subset of the compact $S^2$. 
Moreover, $v|_{S_k}$ is nonconstant on each connected component of $S_k$, since $u_1$ is nonconstant (as it has positive energy) and its critical points in $S^2$ are a finite set by \cite[Lemma~2.4.1]{McDuffSalamon2}.

So, we can apply Lemma~\ref{lem:holmonotone} to $v|_{S_k}:S_k\to V=\R^{2n}$ and the open ball $\mathring{B}_{R_k}=\{q\in\R^n\,|\, \|q\|<R_k\}$ centered at $p=0\in\R^{2n}$ to obtain
$$
\pi R_k^2 \;\leq\; \int_{v^{-1}(\mathring{B}_{R_k})} v^*\omega_{\text{st}} 
\;=\; 
\int_{u_1^{-1}(\phi(\mathring{B}_{R_k}))} u_1^*\omega
\;\leq\;
\int_{S^2} u_1^*\omega
\;=\;\pi (r+\varepsilon)^2 . 
$$
As $R_k\to R'$, this yields
$\pi (R')^2 \leq \pi (r+\varepsilon)^2$ as claimed; and by taking $R'\to R$ and $\eps\to 0$ this proves the non-squeezing $R \leq r$.
\end{proof}

\subsection{A compact moduli space}
\label{sec:moduli-space}

In this and the next subsection, we prove Theorem~\ref{thm-intro:J1-sphere}, while assuming that the M-polyfold construction in \S\ref{sec:polyfold-setup} holds.

Our argument is a special case of proving the independence of Gromov-Witten invariants from the choice of a compatible almost complex structure $J$. Indeed, in Lemma~\ref{lem:only-one-J0-curve} we compute the number of pseudoholomorphic curves in the particular homology class intersecting the given point $p_0$ to be $1$ for $J=J_0$. So, by showing that this count is independent of $J$, we can show the existence of a $J_1$-holomorphic map in Theorem~\ref{thm-intro:J1-sphere}.\\
For that, we use the fact that the space of $\omega$-compatible almost complex structures is contractible (see e.g.~\cite[Prop.~4.1]{McDuffSalamon1}).
Thus, we can choose a smooth path $ (J_t)_{t \in [0,1]}$ of $\omega$-compatible almost complex structures from $J_0$ to $J_1$. 
Moreover, we fix a marked point $z_0\in S^2$ that we require is mapped to $p_0$ by all the considered maps.

Then, for every $t\in[0,1]$, we define the moduli space of $J_t$-holomorphic curves
\begin{align} \label{eq:definition-of-M_t}
    \mathcal{M}_t := \left\lbrace
        \begin{matrix}
            u : S^2 \rightarrow \mathbb{CP}^1 \times T \;\\
            \text{ smooth}
        \end{matrix}
        \Bigg\vert\;
        \begin{matrix}
            u(z_0) = p_0, \;\; \bar{\partial}_{J_t}u =0, \\ [u]= [\mathbb{CP}^1 \times \{\textup{pt}\}]
        \end{matrix}
    \right\rbrace
    \Bigg/ \sim,
\end{align}
where $u \sim u'$ iff there is a biholomorphism $\psi: S^2 \rightarrow S^2$ such that $u' = u \circ \psi$.
Here $\bar{\partial}_{J_t}$ is the Cauchy-Riemann operator for $J_t$, that is, $\bar{\partial}_{J_t}u = \tfrac{1}{2} (\rd u + J_t\circ \rd u \circ i)$.

Now consider the moduli space for the family $(J_t)_{t\in[0,1]}$
\begin{align} \label{eq:definition-of-M}
    \mathcal{M} := \left\lbrace (t,[u]) \;\vert\; t\in[0,1], [u]\in \mathcal{M}_t \right\rbrace.
\end{align}

We will prove Theorem~\ref{thm-intro:J1-sphere} by contradiction: Assuming $\cM_1 = \emptyset$, in \S\ref{sec:applying-regularization-scheme} we will show that a perturbation of $\cM$ is a compact $1$-dimensional cobordism from $\cM_0$ to $\cM_1$. 
However, $\cM_0 = \{[u_0]\}$ consists of exactly one element (see Lemma~\ref{lem:only-one-J0-curve}). Therefore, this cobordism contradicts that $\cM_1$ is empty.

The first step in completing the proof of Theorem~\ref{thm-intro:J1-sphere} is to establish compactness of the unperturbed moduli space. This is a special case of Gromov's compactness theorem, where bubbling is excluded in the given homology class using the restriction of the topology of $T$ from Remark~\ref{rmk:J1-sphere}. 
For general symplectic manifolds $(T,\omega_T)$, we would need to compactify $\cM$ by bubble trees. So, the subsequent proof of $\cM_1$ being nonempty would only show the existence of (possibly singular) $J_1$-curves and not necessarily smooth spheres in the required homology class. 

\begin{thm}
\label{thm:compactness}
Let $(T,\omega_T)$ be a compact symplectic manifold with $\omega_T(\pi_2(T))=0$.
Then the moduli space $\mathcal{M}$ defined in Equation \eqref{eq:definition-of-M} is compact with respect to the quotient topology induced by $[0,1]\times\mathcal{C^\infty}(S^2,M)$.
\end{thm}

\begin{proof}
This proof follows \cite[Chapter~4]{hummel} and \cite{Ackermann}.
Let $(t_n, [u_n])$ be a sequence in the moduli space $\cM $. In particular, $u_n$ is a sequence of $J_{t_n}$-holomorphic maps in $\cM$. We want to show that there exists a subsequence which converges to an element $(t_\infty,[u_\infty]) \in \cM$, meaning that $u_\infty$ is a $J_{t_\infty}$-holomorphic map. Here, by \textit{convergence} we mean the following:
\begin{enumerate}[(1)]
    \item $t_n$ converges to $t_\infty$ in the usual topology of $[0,1]\subset\R$, \label{item:t_convergence}
    \item $[u_n]$ converges to $[u_\infty] \in \cM_{t_\infty}$ in the \textit{Gromov sense}, that is, there exist biholomorphic maps $\phi_n : (S^2, j) \to (S^2, j)$ with $\phi_n(z_0) = z_0$ such that the reparameterized maps $u_n \circ \phi_n : S^2 \to Q$ converge in $C^\infty$ to $u_\infty$. \label{item:u_convergence}
\end{enumerate}

To achieve \ref{item:t_convergence} we can choose a subsequence of $t_n\in[0,1]$ with $t_n$ converging to a $t_\infty \in [0,1]$ since the interval is compact. 
Then, as $\{J_t\}_{t \in [0,1]}$ is a continuous path, we can deduce $\cC^\infty$-convergence of the almost complex structures $J_{t_n}\to J_{t_\infty}$. 
To achieve \ref{item:u_convergence}, we consider this subsequence, denoting it again by $(t_n, [u_n])$. 
The key observation is that the area functional is uniformly bounded on $\mathcal{M}$. Indeed, all $J_{t_n}$-holomorphic curves $u_n$ represent the same homology class $[u] = [\mathbb{CP}^1 \times \{\textup{pt}\}]$, so that
$$E(u_n) = \int_{\mathbb{CP}^1} u_n^{*} \omega = \omega( [\mathbb{CP}^1 \times \{\textup{pt}\}])=\pi(r+\eps)^2$$
is constant and hence bounded. Moreover, $S^2$ is a closed surface.
Thus, by Gromov's compactness theorem (e.g.\ \cite[Chapter V, Thm.~1.2]{hummel}) there exists a subsequence of $[u_n]$ converging in the Gromov sense to a $J_{t_\infty}$-holomorphic cusp curve $\overline{u}_\infty$ of the same energy $E(\overline{u}_\infty)=\pi(r+\eps)^2$.

Actually, this cusp curve consists of a single $J_{t_\infty}$-holomorphic sphere. Indeed, let $[v_1], \dots,[v_k]$ be the non-constant components of $\overline{u}_\infty$ of energy $E(v_n) = \int_{\mathbb{CP}^1} v_n^* \omega >0$ which sum to $E(v_1)+ \ldots + E(v_k)= E(\overline{u}_\infty) = \pi(r+\eps)^2$. 
Since the symplectic form $\omega=\omega_{\mathbb{CP}^1} \oplus \omega_{T}$ splits, the energies are the sums $E(u_n)=\omega_{\mathbb{CP}^1}(\alpha_n) +\omega_T(\beta_n)$ of symplectic areas of the projections $\alpha_n:=[{\rm pr}_{\CP^1}\circ u_n]$ and $\beta_n:=[{\rm pr}_{T}\circ u_n]$ to the factors $\CP^1$ and $T$.
Here we have $\omega_T(\beta_n)=0$ because of the assumption $\omega_T(\pi_2(T))=0$, 
and $\omega_{\mathbb{CP}^1}(\alpha_n)\in\Z\pi(r+\eps)^2$ since $H_2(\CP^1)$ is generated by $[\CP^1]$ which has symplectic area $\pi(r+\eps)^2$ by construction.
Thus each nontrivial component has energy at least $E(v_n)\geq\pi(r+\eps)^2$, but since the total energy of the bubble tree is $\pi(r+\eps)^2$ this implies $k=1$.
This means that the limit cusp curve $\overline{u}_\infty$ has one non-constant component. Since it has only one marked point (arising from $z_0$ in the definition of $\cM_t$), it cannot have ghost components, and thus $\overline{u}_\infty$ consists of a single $J_{t_\infty}$-holomorphic sphere $u_\infty \colon S^2 \to \mathbb{CP}^1 \times T$. 

The meaning of Gromov-convergence $[u_n]\to \overline{u}_\infty = [u_\infty]$ is exactly as stated in (2) above, so we have shown that $\cM$ is sequentially compact. Finally, compactness follows from the fact that the Gromov-topology is metrizable; see \cite[Theorem~5.6.6.]{McDuffSalamon2}.  
\end{proof}

\subsection{Applying the polyfold regularization scheme}
\label{sec:applying-regularization-scheme}

This section proves Theorem~\ref{thm-intro:J1-sphere}.  We will use the notation and facts established in \S\ref{sec:unique-J0-curve} and \S\ref{sec:moduli-space}, the polyfold constructions from \S\ref{sec:polyfold-setup}, and the polyfold regularization scheme from \S\ref{subsec:polyfold}.

\begin{proof}[Proof of Theorem~\ref{thm-intro:J1-sphere}]
Assume that $\mathcal{M}_1=\emptyset$. Under this assumption, we will the polyfold regularization scheme in Theorem~\ref{thm:polyfold-regularization-scheme} to perturb $\mathcal{M}$ just enough to achieve a smooth structure, while not loosing compactness, or changing its boundary at $t=0,1$. With that, we obtain a compact cobordism between $\mathcal{M}_0= \{[u_0]\}$ and $\mathcal{M}_1=\emptyset$. For that purpose we construct the following objects in \S\ref{sec:polyfold-setup}\footnote{
We introduce a candidate space $\cB$ in \eqref{eq:definition-of-B}, then we establish the M-polyfold structure on an open subset $\cB'\subset\cB$, that we rename into $\cB$ here for ease of notation. 
The exact choice of $\cB'$ as discussed in Remark~\ref{rmk:smaller-base-space-no-problem} is immaterial, since the moduli space $\cM$ and all its regular perturbations will be automatically contained in $[0,1]\times\cB'$. 
}: 

\begin{itemize}
    \item an ambient M-polyfold $[0,1] \times \cB \supset \cM$ modeled on sc-Hilbert spaces (Theorem~\ref{thm:polyfold-structure-base-space}),
    \item a tame strong M-polyfold bundle $\cE \rightarrow [0,1] \times \cB$ (Theorem~\ref{thm:polyfold-structure-bundle}),
    \item a sc-Fredholm section $\sigma: [0,1] \times \cB \rightarrow \cE$ such that $\cM = \sigma^{-1}(0)\subset [0,1]\times \cB$ (Theorem~\ref{thm:sc-Fredholm}).
\end{itemize}
 The use of sc-Hilbert spaces guarantees the existence of sc-smooth bump functions on $\cB$ by \cite[\S5.5]{HWZbook}, which is required for Theorem~\ref{thm:polyfold-regularization-scheme}. We prove that $\sigma$ is transverse to the zero section at $\{0\}\times \cB$ in Theorem~\ref{thm:transverslity-at-0}. Moreover, the assumption $\cM_1=\emptyset$ implies transversality of $\sigma$ at $\{1\}\times \cB$ (see Remark~\ref{rmk:transversality-at-1}).
Now we apply the M-polyfold regularization scheme, see Theorem~\ref{thm:polyfold-regularization-scheme}. This gives a perturbation section $p:[0,1]\times\cB \rightarrow\cE$, such that $(\sigma + p)^{-1} =: \cM^p$
is a compact 1-dimensional manifold. By Remark~\ref{rmk:perturbation-supported-inside-interval}, we can assume $p$ to be supported inside $(0,1)\times \cB$. So, the boundary of $\cM^p$ is
\begin{align*}
\partial (\mathcal{M}^p)
 \;=\; \mathcal{M}^p \cap \partial ([0,1]\times\cB) 
 \;=\;\mathcal{M}^p \cap (\{0,1\}\times\cB)
 \;\cong\; \mathcal{M}_0 \sqcup \mathcal{M}_1 \;=\; \{[u_0]\}.
\end{align*}
Thus, the boundary $\partial (\mathcal{M}^p)$ consists of only one point. Such a manifold does not exist. Therefore, the assumption $\mathcal{M}_1 = \emptyset$ was false and we have proven Theorem~\ref{thm-intro:J1-sphere}.
\end{proof}

\begin{rmk} \rm 
\label{rmk:2n+1-dim-cobordism}
If we dropped the condition $u(z_0)=p_0$ in the construction of the moduli spaces $\cM_t$ in Equation \eqref{eq:definition-of-M_t}, then we could directly use the polyfold setup for Gromov-Witten moduli spaces with one marked point from \cite{HWZgw}. Then the above arguments would provide a (2n+1)-dimensional cobordism $\mathcal{M}^p$ between the 2n-dimensional manifolds $\cM_0$ and $\cM_1^p$, where the latter is obtained from $\cM_1$ by the perturbation $p|_{\{1\}\times \cB}$. One would need to choose the perturbation to be supported away from $J_1$-curves intersecting the point $p_0\in \CP^1\times T$ (which by assumption do not exist), so that the evaluation map on the perturbed moduli space $\ev: \cM_1^p\to \CP^1\times T$ does not contain $p_0$ in its image and hence has degree $0$. 
Moreover, one would need to formulate Lemma~\ref{lem:only-one-J0-curve} in a way that the evaluation map $\ev:\cM_0\to\CP^1\times T$ is a bijection and thus has degree $1$.
This difference between degrees is then in contradiction to the fact that the evaluation map $\cB\to\CP^1\times T$ extends both $\ev|_{\cM_1^p}$ and $\ev|_{\cM_0}$ to a continuous map $\ev:\cM^p \to \CP^1\times T$ on the cobordism. 
\end{rmk}

\section{Polyfold setup}
\label{sec:polyfold-setup}

This section provides the polyfold description of the compact moduli space $\cM$ that is the basis of the proof of Theorem~\ref{thm-intro:J1-sphere} in \S\ref{sec:applying-regularization-scheme}.

Recall that $(T,\omega_T,J_T)$ denotes the torus with a compatible pair of symplectic form and almost complex structure (e.g.\ the one constructed in \S\ref{sec:compatifying-target-space}). We may consider any other compact symplectic manifold with ${\omega_T(\pi_2(T))=0}$, as explained in Remark~\ref{rmk:J1-sphere}. 
Moreover, we fix a point $p_0$ and compatible almost complex structure $J$ on $(Q,\omega)=(\mathbb{CP}^1 \times T,\omega_{\CP^1}\times\omega_T)$. We choose a smooth path $ (J_t)_{t \in [0,1]}$ of $\omega$-compatible almost complex structures from $J_0=i\times J_T$ to $J_1=J$.
Then, \S\ref{sec:moduli-space} proves compactness of the family $\mathcal{M} = \left\lbrace (t,[u]) \;\vert\; t\in[0,1], [u]\in \mathcal{M}_t \right\rbrace$ of moduli spaces
\begin{align*} 
    \mathcal{M}_t = \left\lbrace
        \begin{matrix}
            u : S^2 \rightarrow \mathbb{CP}^1 \times T \;\\
            \text{ smooth}
        \end{matrix}
        \Bigg\vert\;
        \begin{matrix}
            u(z_0) = p_0, \;\; \bar{\partial}_{J_t}u =0, \\ [u]= [\mathbb{CP}^1 \times \{\textup{pt}\}]
        \end{matrix}
    \right\rbrace
    \Bigg/ \sim , 
\end{align*}
with the equivalence relation 
\begin{align*}
    u \sim v \;\; :\Longleftrightarrow\;\; 
    \begin{matrix}
        \exists \psi: S^2 \rightarrow S^2 \text{ biholomorphism with } \psi(z_0)=z_0 \text{ and } u = v \circ \psi.
    \end{matrix}
\end{align*}

The polyfold setup starts with a choice of an ambient space that contains $\mathcal{M}$ (as a compact zero set of a sc-Fredholm section). 
For the family $\mathcal{M}=\bigcup_{t\in [0,1]}\cM_t$, the natural choice of ambient space is $ [0,1]\times \cB$, where $\cB$ is an ambient space for each of the moduli spaces $\cM_t$, given by
\begin{align}
    \mathcal{B} := \left\lbrace 
        \begin{matrix}
            u : S^2 \rightarrow \mathbb{CP}^1 \times T\;\\
            \text{ of class }  W^{3,2} 
        \end{matrix}
        \Bigg\vert\;
        \begin{matrix}
            u(z_0) = p_0,\\
            [u]= [\mathbb{CP}^1 \times \{\textup{pt}\}] 
        \end{matrix}
    \right\rbrace 
    \Bigg/ \sim . \label{eq:definition-of-B}
\end{align}
This uses the same equivalence relation $\sim$ as in the definition of $\mathcal{M}_t$, but we equip $\mathcal{B}$ with the quotient topology induced by the metrizable topology on the Hilbert manifold $H=W^{3,2}(S^2 ,\mathbb{CP}^1 \times T)$, unlike the smooth topology in Theorem~\ref{thm:compactness}. The condition $[u]= [\mathbb{CP}^1 \times \{\textup{pt}\}]$ specifies some connected component(s) of $H$, and $u(z_0) = p_0$ cuts out a further submanifold, so that $\cB$ is the quotient of a Hilbert manifold. However, the action by reparameterization with biholomorphisms is not differentiable (see e.g.\ \cite[\S~2.2]{usersguide}), and so $\cB$ does not inherit the smooth strucure of a Hilbert manifold. Instead, we will show in Theorem~\ref{thm:polyfold-structure-base-space} that it carries the structure of a sc-Hilbert manifold.\footnote{
Strictly speaking, our proofs establish the polyfold structures not for $\cB$ and $\cE\to\cB$ as stated, but after restriction to a $W^{3,2}$-open neighbourhood $\cB'\subset\cB$ of the dense subset of smooth curves $\cB_\infty\subset\cB$. 
Additional estimates could prove $\cB'=\cB$, but applications of the polyfold description yield the results for any $\cB_\infty\subset\cB'\subset\cB$ ; see Remark~\ref{rmk:smaller-base-space-no-problem}.
}

We sometimes write an equivalence class as $\alpha \in \cB$ instead of $[u]\in\cB$, indicating that there is no preferred representative in the class.
To build the bundle $\cE\to[0,1]\times\cB$, we consider for each $(t,\alpha)\in[0,1]\times\mathcal{B}$, the Hilbert space quotient
\begin{align}
    \mathcal{E}_{(t,\alpha)} := \left\lbrace (u,\xi) \;\Bigg\vert\;
        \begin{matrix}
            [u] = \alpha \\
            \xi \in       \Lambda^{0,1}_{J_{t}}
            \big(S^2, u^*\rT(\mathbb{CP}^1 \times T)\big) \;   
            \text{of class } W^{2,2}
        \end{matrix}
    \right\rbrace
    \Bigg/ \sim ,  \label{eq:definition-of-fiber}
\end{align}
where $\Lambda^{0,1}_J(S^2, u^*\rT Q)$ denotes the 1-forms on $S^2$ with values in the pullback bundle $u^*\rT Q$ that are complex antilinear with respect to $i$ on $S^2$ and $J$ on $Q=\mathbb{CP}^1 \times T$. The equivalence relation $\sim$ is given by
\begin{align*}
    (u,\xi) \sim (v,\zeta) \;\; :\Longleftrightarrow\;\;
        \begin{matrix}
            \exists \psi: S^2 \rightarrow S^2 \text{ biholomorphism with } \psi(z_0)=z_0 \\ \text{ and } u = v \circ \psi \text{ and } \xi = \zeta\circ \rd\psi.
        \end{matrix}
\end{align*}
Now the bundle $\cE\to[0,1]\times\cB$ is given by the total space
\begin{align}
    \mathcal{E} := \left\lbrace (t,[(u,\xi)]) \;\Big\vert\; t\in[0,1],\alpha\in\cB, [(u,\xi)] \in \mathcal{E}_{(t,\alpha)}   \right\rbrace  \label{eq:definition-of-E}
\end{align}
with the projection to $[0,1]\times\mathcal{B}$. This projection is well-defined since $(u,\xi) \sim (v,\zeta)$ implies $u\sim v$. 

Finally, the section
\begin{align}
    \sigma : [0,1]\times \mathcal{B} \rightarrow \mathcal{E} , \qquad
    (t,[u]) \mapsto (t,[(u,\bar{\partial}_{J_t}u)]) \label{eq:definition-of-sigma}
\end{align}
cuts out the moduli space $\sigma^{-1}(0) = \mathcal{M}$.
To apply the M-polyfold regularization Theorem~\ref{thm:polyfold-regularization-scheme}, we need to equip the spaces $\mathcal{B}$ and $\mathcal{E}$ with sc-smooth structures such that $\sigma$ is sc-smooth, and moreover show that $\sigma$ is a sc-Fredholm section.

\subsection[The Gromov-Witten space of stable curves]{The Gromov-Witten space of stable curves}
\label{sec:B-inside-space-of-stable-curves}

The sc-smooth structure on the base space $\cB$ is obtained in \S\ref{sec:base} by understanding it as a subset $\cB \subset Z$ of a space of 
stable curves in the manifold $Q=\mathbb{CP}^1\times T$.
The space
\begin{align}
Z := \left\lbrace 
            u: S^2 \to Q \text{ of class } W^{3,2} \,\big|\, 
            [u]= [\mathbb{CP}^1 \times \{\textup{pt}\}]
\right\rbrace
    \Big/ \sim
\label{eq:definition-of-Z}
\end{align}
does not satisfy the condition $u(z_0)=p_0$, but the marked point $z_0$ is present in the definition of the equivalence relation $\sim$, that equals the one in \eqref{eq:definition-of-B}.
Thus, there exists a well defined evaluation map 
$\text{ev}: Z \rightarrow Q,\; \text{ev}([u]) := u(z_0)$.

Then, the base space $\cB$ can be viewed as the preimage of the point $p_0\in Q$ under the evaluation map $\text{ev}$, that is
\begin{align}
    \cB \;=\; \left\lbrace [u] \in Z \; \vert \; u(z_0)=p_0 \right\rbrace \;=\; \text{ev}^{-1}(\{p_0\}) \;\subset\; Z  . 
    \label{eq:BsubsetZ}
\end{align}
The space $Z$ is a subspace of a polyfold that Hofer, Wysocki and Zehnder construct in \cite{HWZgw}, which we will denote by $Z^{\text{HWZ}}$. In \cite{HWZgw}, Hofer, Wysocki and Zehnder also construct a strong polyfold bundle $W\rightarrow Z^{\text{HWZ}}$ and a sc-Fredholm section $\bar{\partial}: Z^{\text{HWZ}} \to W$, that cuts out holomorphic curves.
More precisely, for any numbers $g, k\in\N_0$ and nontrivial homology class $A\in H_2(Q)$, the polyfold $Z^{\text{HWZ}}$ has a component\footnote{These components are open and closed but not necessarily connected.} $Z^{\text{HWZ}}_{g,k,A}$ so that 
$$\bar{\partial}^{-1}(0)\cap Z^{\text{HWZ}}_{g,k,A} =\overline{\cM}_{g,k}(A)$$
is the compactified Gromov-Witten moduli space of (possibly nodal) pseudoholomorphic curves in class $A$ of genus $g$, with $k$ marked points. 
In the following, we will explain how considering genus $g=0$, one marked point, and homology class $A:=[\mathbb{CP}^1\times \{\pt\}]\in H_2(Q)$ will provide an identification $Z\cong Z^{\text{HWZ}}_{0,1,A} \subset Z^{\text{HWZ}}$.

In general, the space $Z^{\text{HWZ}}=\{(S,j,M,D,u) | \ldots\}/\sim$ is defined in \cite[Definition~1.4,1.5]{HWZgw} as a set, and given a topology in \cite[\S3.4]{HWZgw}, as a quotient of
\begin{align*}
\left\lbrace 
        \begin{matrix}
            (S,j,M,D,u) 
        \end{matrix}
        \;\Bigg\vert\;
        \begin{matrix}
            (S,j,M,D) \text{ connected nodal Riemann surface}, \\
            u\in W^{3,2,\delta}(S,Q) , \int_C u^*\omega \geq 0 \text{ for each component } C\subset S, \\
        \int_C u^*\omega > 0 \text{ for each non-stable component } C\subset S  
        \end{matrix}
    \right\rbrace .
\end{align*}
Here $(S,j)$ is a (not necessarily connected) Riemann surface, $M\subset S$ is a finite set of marked points, and $D$ is a finite set of nodal pairs in $S$. These nodal pairs are identified in order to obtain a noded Riemann surface -- which is required to be connected. The maps $u \colon S\to Q$ are then required to descend to a continuous map from the noded Riemann surface, and they are required to be of weighted Sobolev class $W^{3,2,\delta}$ on the complement $S\setminus D$ of the nodal points.

Two stable maps $(S,\ldots ,u)\sim (S',\ldots ,u')$ are equivalent if there is a biholomorphism $\psi$ between the corresponding marked noded Riemann surfaces (i.e.,~compatible with $M$ and $D$), such that $u'=u\circ\psi$.
The conditions on the symplectic area in $Z^{\text{HWZ}}$ are known to guarantee that pseudoholomorphic curves of this type are stable, meaning that their isotropy groups are finite. We show in Lemma~\ref{lem:isotropy-finite} that all maps in $Z^{\text{HWZ}}$ have finite isotropy. In fact, Lemma~\ref{lem:no-isotropy-in-our-homology-class} proves that all maps in our particular case $Z^{\text{HWZ}}_{0,1,A}$ have trivial isotropy. 
To begin with, the following Lemma simplifies the description of this space to the description given in Equation \eqref{eq:definition-of-Z}. 

\begin{lem}
\label{lem:Z-equals-Z}
The polyfold $Z^{\text{HWZ}}_{0,1,A}$ for $A=[\mathbb{CP}^1\times \{\pt\}]\in H_2(Q)$ is naturally identified with $Z$.
The same holds true if we replace the torus $T$ in $Q=\CP^1\times T$ by any compact symplectic manifold with $\omega_T\left( \pi_2(T) \right)=0$.
\end{lem}

\begin{proof} 
First recall that the symplectic area of a map $u:S\rightarrow Q$ depends only on its homology class. Thus, for $[u]=[\mathbb{CP}^1\times \{\pt\}]=A\in H_2(Q)$, we obtain $\int_{S}u^*\omega =\omega(A) = \pi(r+\eps)^2>0$, as computed in Equation \eqref{eq:area-of-u0}.
This is the symplectic area of all equivalence classes of (possibly nodal) maps in
\begin{align*}
    Z^{\text{HWZ}}_{0,1,A}
    =
    \left\lbrace 
        \begin{matrix}
            [(S,j,M,D,u)] \in Z^{\text{HWZ}}
        \end{matrix}
        \;\Bigg\vert\;
        \begin{matrix}
            \text{ genus of } (S,j,D) \text{ is } g=0, \; \#M=1, \\
            [u]= A= [\mathbb{CP}^1\times \{\pt\}]
        \end{matrix}
    \right\rbrace.
\end{align*}
We claim that this can be simplified to the formulation given in \eqref{eq:definition-of-Z}. For that, we first recall that all the nodal surfaces of genus $g=0$ are trees of spheres. 
Next, note that the assumption $\omega_T(\pi_2(T))= 0$ guarantees that all components $S^2\simeq C\subset S$ have energy $\int_C u^*\omega \in \omega(\pi_2(Q)) = \omega_{\CP^1}(\pi_2(\CP^1))=\Z \pi(r+\eps)^2$. Recall here that we chose the symplectic form on $\CP^1$ in \S\ref{sec:compatifying-target-space} in such a way that we have $\omega_{\mathbb{CP}^1}([\mathbb{CP}^1]) = 
\int_{\mathbb{CP}^1} \omega_{\mathbb{CP}^1} =\pi(r+\varepsilon)^2$.
As in the proof of Theorem~\ref{thm:compactness}, the homology condition $[u]= [\mathbb{CP}^1 \times \{\textup{pt}\}]$ implies that $S$ can only have one component, on which $u$ is non-constant. Indeed, the total symplectic area
$\int_S u^*\omega = \omega([\mathbb{CP}^1 \times \{\textup{pt}\}])=\pi(r+\varepsilon)^2$ is the sum of non-negative areas of all components, but each non-constant component has energy $\int_C u^*\omega \geq \pi(r+\varepsilon)^2$.
Moreover, $S$ cannot have so-called ghost components $C\subset S$ with $\int_C u^*\omega = 0$ because stability of such components would require at least three special points, while there is only one marked point from $k=1$ and at most one nodal point connecting $C$ to the unique non-constant component.
Thus all nodal surfaces $(S,j,D)$, that are needed for the component $Z^{\text{HWZ}}_{0,1,A} \subset Z^{\text{HWZ}}$, are single spheres, i.e.\ $D=\emptyset$.

The absence of nodes, i.e.~$D=\emptyset$, also explains why we do not need to consider weighted Sobolev spaces but can directly work with maps of Sobolev class $W^{3,2}$.
Moreover, the topology specified in \cite[\S3.4]{HWZgw} simplifies in this setting to the quotient topology coming from the space of $W^{3,2}$-maps.

Finally, we will use the fact that any compact genus $0$ Riemann surface $(S,j)$ without nodes is biholomorphic to $(S^2,i)$, so that for each point in $Z^{\text{HWZ}}_{0,1,A}$ we can choose representatives $[(S^2,i,M,\emptyset,u)]$. The remaining equivalence relation is then by biholomorphisms of $(S^2,i)$, which we can use to fix the marked point $M=\{z_0\}$ and reduce the equivalence relation to reparameterization with biholomorphisms $\psi:S^2 \to S^2$ that fix the marked point $\psi(z_0)=z_0$ as in Equation \eqref{eq:definition-of-B}. 
This identifies the polyfold given in \cite{HWZgw} with $Z$ as defined in Equation \eqref{eq:definition-of-Z} in the following way:
\begin{align*}
    Z^{\text{HWZ}}_{0,1,A}
    &= \left\lbrace 
        \begin{matrix}
            (S,j,M,\emptyset,u)
        \end{matrix}
        \Bigg\vert\;
        \begin{matrix}
            u\in W^{3,2}(S,Q) \\
            \dots\\
            [u]= [\mathbb{CP}^1 \times \{\textup{pt}\}] 
        \end{matrix}
    \right\rbrace
    \Bigg/ 
    \begin{matrix}
    (S,j,M,u) \hfill \\
    \sim (S',\psi^* j, \psi^{-1}(M), u\circ\psi )
    \end{matrix}\\
      &\cong \left\lbrace 
        \begin{matrix}
            (S^2,i,\{z_0\},\emptyset, u)
        \end{matrix}
        \Bigg\vert\;
        \begin{matrix}
            u\in W^{3,2}(S^2,Q) \\
            [u]= [\mathbb{CP}^1 \times \{\textup{pt}\}] 
        \end{matrix}
    \right\rbrace
    \Bigg/ \sim \\
    &\cong \left\lbrace 
            u\in W^{3,2}(S^2,Q) \;
        \big\vert\;
        \begin{matrix}
            [u]= [\mathbb{CP}^1 \times \{\textup{pt}\}] 
        \end{matrix}
    \right\rbrace
    \big/ \sim \quad=\; Z. \qedhere
\end{align*}
\end{proof}

\subsection{Trivial isotropy}

In this section, we show that all (not necessarily pseudoholomorphic) maps in the space $Z$ defined in \eqref{eq:definition-of-Z} have trivial isotropy due to their specific homology class.
We start by showing that maps of nontrivial finite energy have finite isotropy groups for any compact domain and target.  

\begin{lem}
\label{lem:isotropy-finite}
Let $(Q,\omega)$ be any symplectic manifold, $(\Sigma,j)$ a compact connected Riemann surface, and let $u: \Sigma \rightarrow Q$ be a $\mathcal{C}^1$-map with positive symplectic area $\int_{\Sigma} u^*\omega >0$. Then it has a finite isotropy group
\begin{align*}
    G_u 
    := \{ \psi:(\Sigma,j) \rightarrow (\Sigma,j) 
    \textup{ biholomorphic} \;\vert\;
    u\circ\psi =u \} .  
\end{align*}
\end{lem}

\begin{proof}
Having positive symplectic area implies that there exists an open ball $B\subset\Sigma$, so that $u$ is injective on $B$. Indeed, since $\int u^*\omega >0$ there must be a point $p \in \Sigma$ such that $(u^* \omega)_p$ does not vanish as a bilinear form on $\rT_p \Sigma$, i.e.~there are vectors $v,w\in \rT_p \Sigma$ with $(u^* \omega)_p(v,w) > 0$. This is equivalent to $\omega_{u(p)}(\rd u(p)(v), \rd u(p)(w))>0$.
Since $\omega$ is skew-symmetric, we know that $\rd u(p)(v), \rd u(p)(w)$ are linearly independent. Thus $\rd u(p)$ has maximal rank $2$, and we can find a ball $B$ around $p$ such that $u|_B$ is injective and $\int_B u^* \omega >0$.

Next, we claim that the images $g(B)$ of $B$ under the automorphisms $g\in G_u$ are all disjoint. Assume this is not the case. 
Then there exists a $g\in G_u\setminus\{{\rm id}\}$ such that $B\cap g(B)\supset U$ contains a nonempty open set $U$. For every $p\in U$ we have $p=g(q_p)$ for some $q_p\in B$. Since $u\circ g = u$ we have $u(p)=u(q_p)$, so that the injectivity of $u|_B$ implies that $p=q_p$. This shows $g|_U\equiv {\rm id}$, so that unique continuation for biholomorphisms on the connected surface $\Sigma$ implies $g={\rm id}$, contradicting the assumption.

Therefore, $\bigcup_{g\in G_u} g(B)$ is a disjoint union of open sets, and each restriction $u|_{g(B)}$ has the same positive energy 
$$
\int_{g(B)} u^*\omega =\int_{B} g^*u^*\omega = \int_{B} u^*\omega =:\delta >0 . 
$$
If $u^*\omega$ is everywhere non-negative, this implies that $G_u$ cannot have more than $\delta^{-1} \int_{\Sigma} u^*\omega$ elements. This is the case for $u$ being pseudoholomorphic.
To prove finiteness of $G_u$ in general, we pick metrics on $\Sigma$ and $Q$ with respect to which $\rd u$ and $\omega$ are bounded. Then we have $\int_{g(B)} u^*\omega \leq C~ \rm{Vol}(g(B))$ for some constant $C>0$, and hence $\rm{Vol}(g(B))\geq\tfrac{\delta}{C}$. Since the total volume of $\Sigma$ is finite, and the sets $g(B)$ are disjoint, this implies that $G_u$ must be finite.
\end{proof}

For spheres in the specific homology class $[u]=[\CP^1\times\{\pt\}]$ in $Q=\CP^1\times T$ we can extend this argument to show that the isotropy groups are in fact trivial.

\begin{lem}
\label{lem:no-isotropy-in-our-homology-class}
If $[u]\in Z$, then $G_u = \{\textup{id}\}$.
\end{lem}
\begin{proof}
The Sobolev embedding $W^{3,2}(S^2,Q)\subset\cC^1(S^2,Q)$ and Lemma~\ref{lem:isotropy-finite} imply that elements of $Z$ have finite isotropy groups. 
To prove that they are trivial, we consider $u\in W^{3,2}(S^2,Q)$ with finite but nontrivial isotropy group $G_u \neq \{\textup{id}\}$ and we will show that $[u] \neq [\mathbb{CP}^1 \times \{\textup{pt}\}]$, and thus $[u]\notin Z$.

Note that $G_u\subset \textup{Aut}(S^2,i)=\textup{PSL}(2,\C)$ is a subgroup of the Möbius group. Since $G_u$ is finite, it must consist of elements of finite order. 
Möbius transformations are classified into parabolic, elliptic and hyperbolic/loxodromic ones, corresponding to their geometric and algebraic properties.\footnote{See for example the lecture notes \cite{Olsen} for a detailed geometric description of the Möbius group, especially \cite[Cor.~12.1]{Olsen} for the statement about elements of finite order.}
The only Möbius transformations of finite order $k>1$ are elliptic ones corresponding to a rotation by angle $\frac{2\pi}{k}$ around two different fixed points in the extended complex plane. 
Since $G_u$ is assumed to be nontrivial, it must contain some such rotation $f\in\textup{Aut}(S^2,i)$.
We can moreover choose a biholomorphism $\psi$ of $\C \cup \{\infty\}$ that maps the fixed points of the rotation to $0$ and $\infty$. Then the map $u' := u \circ \psi^{-1}$ represents the same homology class as $u$ and its isotropy group contains $g:= \psi\circ f\circ\psi^{-1}$, which is a rotation of order $k>1$ fixing $0$ and $\infty$. 
Thus, $g: \C\cup\{\infty\} \to \C\cup\{\infty\}$ is given by $g(z)= e^{2\pi i/k}z$, and we have $u'(re^{i\theta + m\cdot2\pi i /k}) = u'(re^{i\theta})$ for all $m\in\Z$ since $g^m\in G_{u'}$. 
This allows us to factorize $u' = v \circ \rho_k$ with $v(re^{i\theta}):=u'(r e^{i \theta / k})$
and $\rho_k(r e^{i\theta}) := r e^{ki\theta} $ for $r\in (0,\infty)$, $\theta\in [0,2\pi]$.
By identifying $S^2\cong \C\cup\{\infty\}$, this defines continuous maps $v:S^2\to Q$ and $\rho_k:S^2\to S^2$  with $v(0)=u(0)$, $v(\infty)=u(\infty)$, $\rho_k(0)=0$, and $\rho_k(\infty)=\infty$. 
Finally, this implies $[u]=[u'] ={\rm deg}(\rho_k) \cdot  [v] \in H_2(Q)$, where ${\rm deg}(\rho_k) = k >1$, because $\rho_k$ is a $k$-fold cover of $S^2$.  
This contradicts $[u]=[\CP^1\times\{\pt\}]$, since $\frac 1k [\CP^1]\in H_2(\CP^1)$ is not representable by a map $\pr_{\CP^1}\circ v: S^2 \to \CP^1$.
\end{proof}

\subsection{The base space}
\label{sec:base}

Within this section, we explain how to equip the base space $\cB$ defined in \eqref{eq:definition-of-B} (and thus also $[0,1]\times\cB$ -- see Corollary~\ref{cor:boundary}) with a polyfold structure. Since the isotropy is trivial by Lemma~\ref{lem:no-isotropy-in-our-homology-class}, this means we give $\cB$ an M-polyfold structure, as discussed in \S\ref{subsec:polyfold}.
In fact, due to the absence of nodal maps, we can specify this further to an atlas of local homeomorphisms to open subsets of sc-Hilbert spaces, whose transition maps are sc-smooth. 

\begin{rmk} \rm \label{rmk:sc-filtration}
Before stating this result rigorously, we need to introduce one more piece of polyfold notation (also see \cite[\S4.1]{usersguide}). Every polyfold (and thus also M-polyfold or sc-Hilbert manifold) $Z$ contains a dense subset $Z_\infty\subset Z$ of so-called smooth points and a nested sequence of subsets $$Z_\infty \subset \ldots \subset Z_{k+1} \subset Z_k \subset \ldots Z_0=Z.$$ 
Each of these is equipped with its own metrizable topology, so that, in particular, the inclusion maps $Z_{k+1}\hookrightarrow Z_k$ are continuous. 

In most applications, the smooth points $z\in Z_\infty$ are the smooth maps modulo reparameterization, whose domains may be nodal. 
For the Gromov-Witten polyfold $Z$ in Equation \eqref{eq:definition-of-Z}, the points $[u]\in Z_k$ are given by maps $u:S^2\to Q$ of class $W^{3+k,2}$, and $Z_k$ is equipped with the quotient of the $W^{3+k,2}$ -topology. Correspondingly, $Z_\infty$ consists of the equivalence classes of smooth maps.
\end{rmk}

\begin{thm}
\label{thm:polyfold-structure-base-space}
After replacing the base space $\cB$ with an open neighborhood $\cB'\subset \cB$ of the smooth points $\cB_{\infty} =\cB\cap Z_\infty$, it carries the natural structure of a sc-Hilbert manifold and thus of an M-polyfold.
\end{thm}

\begin{rmk} \rm \label{rmk:smaller-base-space-no-problem}
In practice, we expect $\cB'=\cB$ by an analogue of the estimates in \cite[Theorems~3.8, 3.10]{HWZgw}, which guarantee that charts constructed on neighbourhoods of smooth points $b\in\cB_\infty:=\cB\cap Z_\infty$ cover all of $\cB$.

We need to center charts at smooth points in both proof approaches that we will present. Since we avoid all avoidable estimates, this proves a sc-Hilbert structure on an open neighbourhood $\cB'\subset\cB$ of $\cB_{\infty}$. That is, $\cB'$ contains all equivalence classes of smooth maps $u:S^2\to Q$. (In fact, $\cB'$ contains a $W^{3,2}$-neighbourhood of each such smooth point $[u]$).
However, we may have $[v]\in \cB \setminus\cB'$ for some $v\in W^{3,2}\setminus\cC^\infty$.
While this makes the description of the base space less explicit, it does not affect the rest of the proof.

In fact, we could even allow for a more drastic restriction of the base space $\cB'\subset\cB$ as follows.  Proving Theorem~\ref{thm:polyfold-structure-base-space}, by applying \cite[Thm.~5.10~(I)]{ben} directly, establishes an M-polyfold structure on a $Z_1$-open subset $\cB'\subset \cB\cap Z_1$ such that $\cB\cap Z_\infty \subset \cB'$. 
This would remove all maps in $W^{3,2}\setminus W^{4,2}$ from $\cB'$ and guarantee only that $\cB'$ contains a $W^{4,2}$-neighbourhood of smooth maps.

However, the moduli spaces $\cM_t$ and their perturbations automatically lie in the $\infty$-level $\cB_{\infty}$ due to the regularizing property of sc-Fredholm sections \cite[Def.3.8]{HWZbook}. So neither a shift to $Z_1$-topology nor restricting to a neighbourhood of $\cB_{\infty}$ affects how we can use the M-polyfold regularization scheme (see Theorem~\ref{thm:polyfold-regularization-scheme}) in the proof of Theorem~\ref{thm:nonsqueezing}.
\end{rmk}

There are several ways to prove Theorem~\ref{thm:polyfold-structure-base-space}. 
We will first explain how the natural\footnote{The choices made in the general construction for Gromov-Witten spaces, of a ``gluing profile'' and sequence $2\pi> \ldots \delta_{m+1}>\delta_m > \ldots >0$ of exponential weights, turn out to be irrelevant in our special case due to the absence of nodes.} polyfold structure for
$Z \subset Z^{\text{HWZ}}$ constructed in \cite{HWZgw} induces a polyfold structure on its subset $\cB= \text{ev}^{-1}(\{p_0\}) \subset Z$. 
This proof applies an implicit function theorem by Filippenko~\cite{ben} to the evaluation map $\ev:Z\to Q$. 

\begin{proof}[Proof of Theorem~\ref{thm:polyfold-structure-base-space} by Implicit Function Theorem]
In \cite{HWZgw}, the space $Z^{\text{HWZ}}$ and thus also its (open and closed) component $Z$ is given a polyfold structure.
Following the proof in \cite{HWZgw}, one sees that for the component $Z$ considered here (i.e.\ genus $0$, homology class $[\mathbb{CP}^1 \times \{\textup{pt}\}]$), the model spaces are sc-Hilbert spaces (i.e.\ all retractions are identity maps). Their scale structure corresponds to the dense subsets $Z_m = \{[u] \in Z\;\vert\; u\in W^{3+m,2}(S^2,Q)\}\subset Z$.
Moreover, by Lemma~\ref{lem:no-isotropy-in-our-homology-class} we have trivial isotropy. So \cite{HWZgw} actually constructs $Z$ as a sc-Hilbert manifold.

Now we will use the description of $\cB= \text{ev}^{-1}(\{p_0\})$ in Equation \eqref{eq:BsubsetZ} as a preimage of $p_0\in Q$ under the evaluation map $\ev:Z \to Q, [u] \mapsto u(z_0)$. 
Note that the evaluation map is classically smooth on each level $Z_m\subset Z$, so it is sc-smooth\footnote{The manifold $Q$ is finite-dimensional and so carries the constant scale structure where $Q_m=Q$ for all $m$.} by \cite[Cor.~1.1]{HWZbook}. Furthermore, \cite[\S5.1]{ben} explains why the evaluation map is transverse (in the sense of \cite[Def.~5.9]{ben}) to every submanifold of $Q$, so in particular transverse to $\{p_0\}\subset Q$. Now we can apply \cite[Thm.~5.10~(I)]{ben} to deduce that $\cB= \text{ev}^{-1}(\{p_0\})$ is an M-polyfold.
Though, strictly speaking, \cite{ben} constructs an M-polyfold structure on an open neighbourhood of the $\infty$-level $\cB_{\infty}:=\cB \cap Z_{\infty}$ in the 1-level $\cB \cap Z_1$.
The proof of nonsqueezing could work with such a neighbourhood (see Remark~\ref{rmk:smaller-base-space-no-problem}), but we will argue that the proof of \cite[Thm.~5.10(I)]{ben} in our specific setting does not actually require a shift in the topology -- just a restriction to a neighbourhood $\cB'\subset\cB$ of $\cB_\infty$. 

By \cite[Rmk.1.4~(ii)]{ben}, the linear model for $\cB$ near $[u]\in\cB$ is given by the kernel of the differential
$\rD\,\text{ev} \, ([u]): T_{[u]}Z \rightarrow T_{p_0}Q$, 
which we need to equip with a sc-structure.
This yields two reasons for the restrictions in \cite{ben}: 
First, the tangent space $T_{[u]}Z$ of an M-polyfold carries a sc-structure only at $[u]\in Z_{\infty}$. 
This is why \cite{ben} builds charts centered at smooth points only, and so do we.
Second, the differential of a general sc-smooth map $Z\to Q$ is defined only at points in the 1-level $Z_1$.
In our case, the evaluation map $\text{ev}$ is classically smooth on $Z_0$, so there is no need to restrict to $[u]\in Z_1$ when using the differential $\rD\,\text{ev} \, ([u])$.

Moreover, the construction of local charts for $\cB$ uses a local submersion normal form, which \cite[Lemma 2.1]{ben} guarantees only if the map is $\cC^1$ (see \cite[Rmk.1.4~(iii)]{ben}). 
Since \cite{ben} considers general sc-smooth maps $Z\to Q$, this requires a restriction to the $\cC^1$-map $Z_1\to Q$. 
In our case there is no need for this restriction, since the evaluation map $\text{ev}:Z_0\to Q$ is $\cC^1$ without a shift. 

So the charts for $\cB$ constructed by \cite{ben} are homeomorphisms between open neighbourhoods of smooth points $[u]\in\cB_\infty$ and open subsets of the kernel of the differential $\rD\,\text{ev}([u])$. (No retractions appear here due to the absence of nodes in $\cB$.)
Since these kernels are sc-Hilbert spaces, this induces a sc-Hilbert manifold structure on the subset $\cB'\subset\cB$ that is covered by the charts.
\end{proof}

Another way to construct the sc-Hilbert manifold structure on $\cB$ is to directly incorporate the condition $u(z_0)=p_0$ in the construction of charts from \cite{HWZgw}. Going through this proof should also serve to illuminate the general approach of \cite{HWZgw} in this simplified setting.

\begin{proof}[Proof of Theorem~\ref{thm:polyfold-structure-base-space} by construction of charts] 
To begin, one needs to check that $\cB$ is a metrizable space. In general, this is proven in \cite[Thm.3.27]{HWZgw} using the Urysohn criteria (Hausdorff, second countable, and completely regular) that imply metrizability. These criteria are easily checked in our setting: First note that the topology on $W^{3,2}(S^2,Q)$ can be obtained by viewing it as a subset of the Hilbert space $W^{3,2}(S^2,\R^N)$ via some choice of embedding $M\subset\R^N$. Now $\cB=\widehat\cB/{\rm Aut}(S^2,i,z_0)$ is the quotient of a subset $\widehat\cB\subset W^{3,2}(S^2,Q)$ (given by specifying the homology class and value $u(z_0)=p_0$), modulo reparameterization by the biholomorphisms in ${\rm Aut}(S^2,i,z_0)$ that fix $z_0\in S^2$. The relative topology on the subset $\widehat\cB$ is automatically metric (thus Hausdorff) and second countable, however not all these properties are inherited by the quotient. 
The metric induces a pseudometric on the quotient, which implies that the quotient topology is completely regular. To show that the quotient is Hausdorff and second countable, it suffices to prove that the quotient map $\pi:\widehat\cB\to\cB$ is open. To check this we consider any open subset $\widehat U\subset\widehat\cB$ and show that $\pi(\widehat U)$ is open by checking that any $u\in\pi^{-1}(\pi(\widehat U))$ has an open neighbourhood contained in $\pi^{-1}(\pi(\widehat U))$. 
Note that $u\circ\psi=\widehat u \in \widehat U$ for some $\psi\in{\rm Aut}(S^2,i,z_0)$ and $\{ w \in \widehat\cB \,|\, \| w - \widehat u \|_{W^{3,2}} <\eps\} \subset \widehat U$ for some $\eps>0$ since $\widehat U$ is open. Now let $C>0$ be a constant bounding all derivatives of $\psi$ up to third order, then we claim that the $\frac{\eps}{C}$ ball around $u$ is contained in $\pi^{-1}(\pi(\widehat U))$. Indeed, $\| v - \widehat u \|_{W^{3,2}} < \frac{\eps}{C}$ implies $\| v\circ\psi - \widehat u \|_{W^{3,2}} <\eps$, thus $v\circ\psi\in U$ and $v\in \pi^{-1}(\pi(\widehat U))$. 
This proves that the quotient map is open and thus finishes the proof of metrizability. 

Next, the main technical work is the construction of a chart for a neighbourhood of a given point $\alpha\in\cB$. For that purpose we pick a representative $u:S^2\to Q$ of $\alpha=[u]$.  
Since $\cB$ is a quotient by the reparameterization action, the charts are constructed as local slices to this action, which involves choices of additional marked points and transverse constraints. More precisely, we need to choose \textbf{good data centered at $u$} in the sense of \cite[Def.~3.6]{HWZgw}.
Such good data exists by \cite[Prop.~3.7]{HWZgw}.
Recall here that the isotropy group $G=G_u=\{\text{id}\}$ is trivial in our case.
Then good data in our setting consists of the following objects with the following properties\footnote{We essentially use the numbering of \cite[Def.~3.6]{HWZgw}, but merged (5) into 8.), merged (7) into 2.)\ and 4.), merged (8) into 7.), and left out (9), (10) which are trivially satisfied in our case.}

\smallskip\noindent
{\bf 1.)} \textbf{Marked points stabilizing the surface:} 
These exist by \cite[Lemma 3.2]{HWZgw}. In our case a stabilization of 
$(S^2,i,\{z_0\},\emptyset)$ for the map $u$ consists of two points 
$\Sigma=\{z_1,z_2\}\subset S^2$ that satisfy the following conditions:
        \begin{itemlist}
            \item $z_0, z_1, z_2 \in S^2$ are pairwise different.
            \item Denote $p_1:=u(z_1), p_2=u(z_2)$.
                  Then $p_0, p_1, p_2$ are pairwise different.
            \item For $i=1,2$, the map $\rd u(z_i) $ is injective, the bilinear form $u^*\omega (z_i)$ is non-degenerate, and it determines the correct orientation on $\rT_{z_i}S^2$.
        \end{itemlist}
    Now the Riemann surface with the additional marked points $(S^2,i,\{z_0,z_1,z_2\})$ is stable; in fact its isotropy group is trivial in our case.
    \cite{HWZgw} also requires a choice of good uniformizing family parameterizing variations of the surface and marked points, but in our case, since the Deligne-Mumford space of three marked points on a sphere is trivial, this family is constant. It remains to choose {\bf small disk structures}, that is holomorphic embeddings of the closed disk $D^2 \simeq D_{z_i}\subset S^2$ with center $0\simeq z_i$ that are disjoint for $i=1,2$. 

\smallskip\noindent
{\bf 2.)} \textbf{Charts for the target space:} Open neighbourhoods $U(p_i)\subset Q$ of $p_i$ for $i=1,2$ and diffeomorphisms $\psi_i :(U(p_i),p_i) \rightarrow (\R^{2n},0)$ are chosen so that
\begin{itemlist}
\item $U(p_1)$ and $U(p_2)$ are disjoint;
\item $u\vert_{D_{z_i}}$ is an embedding for $i=1,2$, with image contained in $U_1(p_i)$. 
\end{itemlist}
Here and in the following we denote by $U_{\rho}(p_i):=\psi_i^{-1}(\{x\in\R^{2n}\,|\, |x|<\rho\})$ the preimages of balls of any radius $\rho>0$.

\smallskip\noindent
{\bf 3.)}  \textbf{A Riemannian metric} on $Q$ is chosen such that it agrees on the open sets $U_4(p_i)$ with the pullback of the standard metric on $\R^{2n}$ by $\psi_i$.

Moreover, we choose an \textbf{open neighbourhood of the zero-section} $\tilde{\cO}\subset \rT Q$, which is fiberwise convex and such that for every $q\in Q$, the exponential map induced by the chosen metric, $\exp: \tilde{\cO}_q := \tilde{\cO} \cap \rT_qQ \rightarrow Q$ is an embedding.
            
\smallskip\noindent
{\bf 4.)} \textbf{Transverse hypersurfaces:}
We choose submanifolds $M_{p_i}\subset U(p_i) \subset Q$ of codimension $2$ for $i=1,2$
such that $p_i\in M_{p_i}$ and 
\begin{itemlist}
\item $\psi_i(M_{p_i}) \subset \R^{2n}$ is a linear subspace; 
\item $\rT_{p_i}Q = \im\, \rd u (z_i) \oplus H_{p_i}$ for $H_{p_i}:=\rT_{p_i} M_{p_i}$; 
\item $\{z_i\} = D_{z_i} \cap u^{-1}(M_{p_i})$ is the only point in $D_{z_i}$ that $u$ maps to $M_{p_i}$. 
\end{itemlist}
This in particular implies that $u|_{D_{z_i}}$ is transverse to $M_{p_i}$ for $i=1,2$.

 \smallskip\noindent
{\bf 6.)} \textbf{Concentric subdisks:} We choose $SD_{z_i}\subset D_{z_i}$ for $i=1,2$ to be the image of smaller disks under the holomorphic embeddings $D^2\simeq D_{z_i}$.

\smallskip\noindent
{\bf 7.)} \textbf{An open neighbourhood} $U\subset W^{3,2}(S^2,u^*\rT Q)$ of $0$ is chosen such that 
\begin{itemlist}
\item 
every section $\eta\in U$ takes values in $u^*\tilde{\cO}$;
\item
for every $\eta\in U$ and $i=1,2$, the map $u':= \exp_u(\eta): S^2 \rightarrow Q$ satisfies $u'(D_{z_i}) \subset U_2(p_i)$, and $u'\vert_{D_{z_i}}$ is an embedding transverse to $M_{p_i}$ that intersects $M_{p_i}$ at a single point $p'_i=u'(z'_i)$, the preimage of some $z'_i\in SD_{z_i}$.
\end{itemlist}

\smallskip

From now on we will assume that we chose good data for a smooth map $u\in\cC^\infty$, since it suffices to construct charts centered at smooth points $\alpha\in\cB_\infty$.\footnote{\cite[Thms.3.8, 3.10]{HWZgw} imply that for every $[u']\in Z$ there is a smooth $u: S^2\rightarrow Q$ and good data centered at $u$ such that $u'= \exp_u(\eta)$ for some $\eta \in U$. This means that the charts coming from good data centered at the smooth points of $Z$ will cover all of $Z$. This can be extended to include the conditions $u'(z_0)=p_0$ and $\eta(z_0)=0$, and thus prove the same for $\cB$.} 
Now the only place where our constructions for $\cB=\{[u'] \,|\, u'(z_0)=p_0 \}\subset Z$ differ from the constructions in \cite{HWZgw} for $Z$ is the definition of a linear subspace of the sc-space $\big(W^{3+k,2}(S^2,u^*\rT Q)\big)_{k\in\mathbb{N}_0}$. Our choice
\begin{align}
    E_u := \{\eta \in W^{3,2}(S^2,u^*\rT Q) \;\vert\; \eta(z_0)=0, \eta(z_i)\in H_{p_i} \text{ for } i=1,2 \} \label{eq:definition-of-E_u}
\end{align}
adds the condition $\eta(z_0)=0$ that linearizes the condition $u'(z_0)=p_0=u(z_0)$. 
We then consider its open subset $\cO:= E_u \cap U$ and claim that the map
\begin{align}
    \cO \to  \cB , \quad \eta \mapsto [\exp_u \eta]  \label{eq:definition-of-chart}
\end{align}
is a homeomorphism
onto a neighbourhood of $[u]$.
While this is not explicitly stated in \cite[\S3.1]{HWZgw}, we can check the properties of this map in our case: 

\smallskip\noindent
\textbf{Continuity:} \eqref{eq:definition-of-chart} is the composition of two continuous maps:  the pointwise exponential map and a quotient projection.
The latter is continuous by definition of the quotient topology on $\cB$. The first map is the composition $\eta \mapsto E \circ \eta$ with the smooth exponential map $E:u^*\rT Q \to Q$. Checking that this is continuous between $W^{3,2}$-topologies requires local estimates, which hold since $3\cdot 2>\dim S^2$ (see e.g.\ \cite[Lemma~B.8]{W-Uh}).

\smallskip\noindent
\textbf{Injectivity:} Consider $\eta,\tilde{\eta}\in \cO$ with $[\exp_u \eta]=[\exp_u \tilde{\eta}]\in\cB$. This means that there is a biholomorphism $\psi: S^2 \rightarrow S^2$ with $\psi(z_0)=z_0$ and $\exp_u \eta = \exp_u \tilde{\eta} \circ \psi$. By property 4.), the maps $\exp_u \eta$ and $\exp_u \tilde{\eta}$ intersect the submanifolds $M_{p_i}$, $i=1,2$ in unique points $p'_i, \tilde{p}_i \in M_{p_i}$ with unique preimages $z'_i,\tilde{z}_i \in SD_{z_i}$. It follows that $\psi(z'_i)=\tilde{z}_i$ for $i=1,2$. On the other hand, $\eta(z_i),\tilde{\eta}(z_i) \in H_{p_i}=\rT_{p_i}M_i$ and the fact that $M_i\subset U(p_i)$ is totally geodesic imply $\exp_u \eta \in M_i$ and $\exp_u \tilde{\eta}\in M_i$, so the uniqueness of intersection points in property 7.) implies $z'_i = z_i$ and $\tilde{z}_i=z_i$ for $i=1,2$. Thus we have $\psi(z_i)=z_i$ for $i=0,1,2$, which implies $\psi = \id_{S^2}$. From this we deduce $\exp_u \eta = \exp_u \tilde{\eta}$ and thus $\eta = \tilde{\eta}$.

\smallskip\noindent
\textbf{Continuity of the inverse:} 
To show that the map in \eqref{eq:definition-of-chart} is a homeomorphism, it remains to check that it maps open subsets of $E_u$ to open subsets of $\cB$. So we fix some $\eta\in\cO$ and need to show that any $\alpha\in\cB$ sufficiently close to $[\exp_u \eta]$ can be written as $\alpha=[\exp_u \eta']$ for some $\eta'\in\cO$. First note that $v:=\exp_u \eta$ satisfies $v(z_0)=p_0$ along with the slicing conditions $v(z_i)\in M_{p_i}$ for $i=1,2$. Next, $\alpha$ being close to $[v]$ in the quotient topology of $\cB$ means that $\alpha=[v']$ for some representative $v'$ that is $W^{3,2}$-close to $v$. By construction of $\cB$, this map satisfies $v'(z_0)=p_0$.
Moreover, since $v$ is locally transverse to the slicing conditions (as specified in 4.), we will have $v'(z'_i)\in M_{p_i}$ for some $z'_i\approx z_i$. Now we can compose $v'$ with a small M\"obius transformation that fixes $z_0$ and maps $z_i$ to $z_i'$ to obtain a new representative $\alpha=[w]$ that satisfies $w(z_0)=p_0$ and $w(z_i)\in M_{p_i}$. This adjustment in slicing conditions guarantees that $\eta'(z):= \exp_{u(z)}^{-1}(w(z))$ defines a section $\eta'\in E_u$. Moreover, the construction is done such that $w$ is still $W^{3,2}$-close to $v$, which guarantees $\eta'\approx\eta$ and thus $\eta'\in\cO$ for $\alpha$ sufficiently close to $[v]$. This proves openness and thus continuity of the inverse.

\smallskip
Thus we have constructed $\cC^0$-charts for $\cB$ centered at any point $[u]$. It remains to equip the local models with sc-structures and show that the transition maps between these charts are sc-smooth.
This is where we (just as \cite[\S3.2]{HWZgw}) have to restrict to centering our charts at points $[u]\in\cB_\infty$ represented by smooth maps $u:S^2\to Q$. This regularity is required to give the pullback bundle $u^*\rT Q$ a classically smooth structure, so that we can define the Sobolev spaces  $W^{3+k,2}(S^2,u^*\rT Q)$ of sections by closure of the smooth sections.
Then we obtain a sc-structure on the Hilbert space $E_u$ in \eqref{eq:definition-of-chart} by the subspaces $E_u\cap W^{3+k,2}(S^2,u^*\rT Q)$ for $k\in\N_0$.

Finally, compatibility of the charts follows directly from \cite{HWZgw}, since our transition maps are the same maps as theirs, just restricted to subsets of their domains.
More precisely, transition maps between different charts for $Z$ can be obtained from its ep-groupoid description by composing local inverses of the source map with the target map. Both of these structure maps are local sc-diffeomorphisms by the \'etale property in \cite[Def.7.3]{HWZbook}  (for the Gromov-Witten case this is established in \cite[Prop.3.19]{HWZgw}). So sc-smoothness of transition maps follows since they are compositions of local sc-diffeomorphisms. 
\end{proof}

\begin{rmk}[sc-smoothness] \rm 
The key point why sc-smoothness appears here is the following: The choice of stabilization points $z_1,z_2$ above depended on $u$. For different $u'$ we might have other $z'_1,z'_2$, and so the transition map between the charts centered at $u$ and $u'$ needs to reparameterize all the vector fields $\eta$. But the reparameterization biholomorphism is not fixed for one transition map, but depends also on the vector field. So we get a map of the form
\begin{align*}
    \cO \ni \eta \longmapsto \eta \circ \rd\psi_{\eta} \in \cO',
\end{align*}
which is not classically differentiable w.r.t.\ any of the usual Sobolev or H\"older norms. It is however scale-smooth; see \cite[\S2.1,2.2]{usersguide} for further discussion.
\end{rmk}

\begin{rmk}[Good data centered at $u_0$]  \rm 
\label{rmk:good-data-at-u_0}
 For the holomorphic map $u_0(z)=(z,m_0)$ from \S\ref{sec:unique-J0-curve}, it is easy to find good data: In the first factor, $u_0$ is the identity, and in the second factor it is the constant map to $m_0\in T$. For this map, any choice of two different points $z_1,z_2\in S^2\setminus\{z_0\}$ yields a stabilization as required. Then $M_{p_i} := \{z_i\} \times T$ are transverse hypersurfaces since their tangent spaces $H_{p_i} = \{0\} \times \rT_{m_0} T$ satisfy
\begin{align*}
    \im\; \rd u_0 (z_i) \oplus H_{p_i} = \rT_{z_i}\CP^1 \oplus \rT_{m_0} T  = \rT_{u_0(z_i)}Q.
\end{align*}
This choice of charts will be used in \S\ref{sec:transversality-at-boundary}. In particular, we will use
\begin{align}
    E_{u_0} := \{\eta \in W^{3,2}(S^2,u^*\rT Q) \;\vert\; \eta(z_0)=0, \eta(z_i)\in \{0\} \times \rT_{m_0} T \text{ for } i=1,2 \}. \label{eq:definition-of-E_u0}
\end{align}
\end{rmk}

\begin{cor}
\label{cor:boundary}
The product $[0,1] \times \cB$ is a sc-Hilbert manifold with boundary $\partial ([0,1] \times \cB) =\{0,1\} \times \cB$.
\end{cor}
\begin{proof}
As a finite dimensional manifold, the interval $[0,1]$ is trivially a sc-Hilbert manifold and the notion of boundary is the same as the classical notion (see also Remark~\ref{rmk:finite-dim-mfd-sc-Hilbert-mfd}).
This means that the product $[0,1]\times \cB$ also is a sc-Hilbert manifold.
In fact, for every pair $(t_0,[u])\in [0,1] \times \cB$ we can choose an open interval $I_{t_0}\subseteq [0,1]$ containing $t_0$ and a chart $\cO\to\cB$ centered at $[u]$ as in \eqref{eq:definition-of-chart}, then a sc-Hilbert manifold chart for a neighbourhood of $(t_0,[u])$ is given by
\begin{align*}
    I_{t_0} \times \cO \;\longrightarrow\;  [0,1] \times \cB  , \qquad
    (t,\eta) \;\longmapsto\; (t,[\exp_u \eta]).
\end{align*}
Sc-smoothness of transition maps between these charts follows directly from sc-smooth compatibility of the charts for $\cB$. 
Finally, $\cO$ has no boundary, and boundary of $I_{t_0}$ arises only from $\partial [0,1]=\{0,1\}$. Since boundary and corner stratifications are determined in local charts, this proves the claim.
\end{proof}

\subsection{The bundle}
\label{sec:bundle}

The purpose of this section is to give the projection $\cE \rightarrow [0,1] \times \cB$ defined in \eqref{eq:definition-of-E} the structure of a tame strong polyfold bundle. To achieve this, we first describe the bundle $W \to Z^{\text{HWZ}}$ from \cite{HWZgw} restricted to $Z\subset Z^{\text{HWZ}}$.

The fibers of $W$ are defined in \cite[\S1.2]{HWZgw} with respect to a fixed almost complex structure $J$ on the target manifold. Using the previous simplifications for $Z\subset Z^{\text{HWZ}}$ (i.e.\ genus $0$, $S=S^2$, homology class $[\mathbb{CP}^1 \times \{\textup{pt}\}]$) we see that the fiber $W_{\alpha}$ over $\alpha\in Z$ is the quotient space
\begin{align*}
    W_{\alpha} = \left\lbrace (u,\eta) \;\Bigg\vert\;
        \begin{matrix}
            [u] = \alpha \\
            \eta \in \Lambda^{0,1}_J\big(S^2, u^*\rT(\mathbb{CP}^1 \times T)\big)\; 
            \text{of class } W^{2,2}
        \end{matrix}
    \right\rbrace
    \Bigg/ \sim 
\end{align*}
of complex antilinear 1-forms of class $W^{2,2}$ with values in the pullback bundle along a representative $u$. Here complex antilinearity of $\eta$ is required with respect to a fixed almost complex structure $J$ on $\mathbb{CP}^1 \times T$ and the standard complex structure on $S^2$.
The equivalence relation is given by
\begin{align*}
    (u,\eta) \sim (v,\mu) \; :\Longleftrightarrow \;
        \begin{matrix}
            \exists \psi: S^2 \rightarrow S^2 \text{ holomorphic} \\ \text{ with } u = v \circ \psi \text{ and } \eta = \mu\circ \rd\psi.
        \end{matrix}
\end{align*}
For $\alpha\in\cB\subset Z$ we can choose a representative $[u]=\alpha$ with $u(z_0)=p_0$, 
and restricting to such representatives reduces the equivalence relation to biholomorphisms $\psi$ with $\psi(z_0)=z_0$.
So if we choose $J = J_t$ for some $t\in[0,1]$, then this identifies our fibers $\cE_{(t,\alpha)} = W_{\alpha}$ with the fibers of a tame strong polyfold bundle constructed in \cite{HWZgw}.
However, we wish to simultaneously extend and restrict the base space: We extend by allowing the almost complex structure to vary, and we restrict to curves through the fixed point $p_0\in Q$. 

\begin{thm}
\label{thm:polyfold-structure-bundle}
Let $\cB'\subset\cB$ be as in Theorem~\ref{thm:polyfold-structure-base-space}\footnote{Recall we expect $\cB'=\cB$ by Remark~\ref{rmk:smaller-base-space-no-problem}.}, then 
$\cE|_{[0,1] \times\cB'} \rightarrow [0,1] \times \cB'$ is a tame strong M-polyfold bundle.
\end{thm}

There are again several ways to prove this. Filippenko \cite{ben} has a result about the restriction of bundles to sub-polyfolds which we explain in Remark~\ref{rmk:proof a}. 
Unfortunately, this would require an existing polyfold description for Gromov-Witten moduli spaces with varying $J$, which we discuss in Remark~\ref{rmk:graphtrick}. 
In any case, the proof of transversality of the section $\sigma$ at $t=0$ requires a fairly explicit bundle chart, so our actual proof is an adaptation of \cite{HWZgw} -- to our simplified setting, but extending the constructions to varying $J$.  

\begin{rmk} \label{rmk:graphtrick}\rm 
The analysis in \cite{HWZgw} is formulated for a fixed almost complex structure $J$ on the target manifold $Q$. This might be deemed as sufficient due to the following graph trick: 

Given a smooth family $(J_t)_{t\in\R^k}$ of almost complex structures, parameterized by a finite dimensional space $\R^k$, we can identify the moduli space of $J_t$-holomorphic curves in $Q$ for some $t\in\R^k$ with a moduli space of certain pseudoholomorphic curves in the product manifold $\widetilde{Q}:=\C^k\times Q$ as follows.
We define an almost complex structure $\widetilde J$ on $\widetilde Q$ by $\widetilde J(t+is,q) := J_{\C^k} \times J_t$ at the point $(t+is,q)\in \widetilde Q =\C^k\times Q$. Then $\widetilde J$-holomorphic maps $\widetilde u: S \to \widetilde Q= \C^k\times Q$ in a class $\widetilde A:=[\{\pt\}]\times A$, $A\in H_2(Q)$ are constant in $\C^k$, so that a constraint $\widetilde u(z_0)\in\R^k\times Q$ at a marked point $z_0\in S$ picks out the maps $\widetilde u$ that are of the form $z\mapsto(t, u(z))$ for some $t\in\R^k$ and a $J_t$-holomorphic map $u:S\to Q$.
Thus pairs $(t,u)$ of $J_t$-holomorphic maps in class $A$ for some $t\in\R^k$ can be identified with $\widetilde J$-holomorphic maps in class $\widetilde A$ satisfying the point constraint. 

Our variation of almost complex structures $(J_t)_{t\in[0,1]}$ can be formulated in this way by choosing $J_t$ to be constant near $t=0$ and $t=1$, so that its constant extension to $t\in\R$ is smooth. However, the polyfold bundle $\widetilde W \to \widetilde Z$ that \cite{HWZgw} constructs for the almost complex manifold $(\widetilde Q, \widetilde J)$ contains already in its base many more (reparameterization classes of) maps $\widetilde u: S^2 \to \C\times Q$ than those of the form $\widetilde u(z)=(t, u(z))$. The point constraint $\widetilde u(z_0)\in\R\times Q$ does not force the $\C$-component to be constant. Similarly, the fibers of $\widetilde W$, consisting of anti-linear 1-forms with values in $\C\times Q$, contain an extra factor $\Lambda^{0,1}(S,\C)$ compared with the fibers $W_\alpha$ above that are used in the polyfold description of a moduli space for fixed $J$.
These infinite dimensional extra factors would have to be split off near maps of the form $(0, u(z))$ or $(1, u(z))$ in order to relate the moduli spaces for $J_0$ and $J_1$ with parts of the moduli space for $\widetilde J$ in this setup. 

Alternatively, extending such a splitting along $(t, u(z))$ for all $t\in\R$ would yield a smaller polyfold description in which the base space consists only of maps of the form $(t,u(z))$. Rather than attempting such a splitting construction, we will directly construct the resulting polyfold bundle.
\end{rmk}

\begin{proof}[Proof of Theorem~\ref{thm:polyfold-structure-bundle}.]
We construct bundle charts as in \cite[\S6.3]{HWZgw}\footnote{Notation here is the same as in \cite{HWZgw}, with lots of simplifications because we work with a constant Riemann surface and trivial isotropy. On the other hand, we allow the almost complex structure to vary in a $1$-parameter family, whereas \cite{HWZgw} fixes it.}.
Fix a pair $(t_0,[u])\in[0,1]\times\cB_\infty$, a representative $u$ of $[u]$, an open interval $I_{t_0}\subset [0,1]$ around $t_0$ (whose `sufficiently small' choice will be specified below), and good data centered at $u$. These choices determine a chart for the base space as in \S\ref{sec:base}, by
$$
I_{t_0} \times \cO \rightarrow [0,1] \times \cB, \qquad
(t,\eta) \mapsto (t, [\exp_u(\eta)]). 
$$
(See \S\ref{sec:base} for the definition of $\cO = E_u \cap U \subset W^{3,2}(S^2,u^*\rT Q)$.)
To build charts for the bundle, recall
that we abbreviate $Q:= \mathbb{CP}^1\times T$. Then, using the chosen representative $u$, the bundle fiber $\cE_{t_0,[u]}$ can be identified with the Hilbert space 
\begin{align}
    F := \Lambda^{0,1}_{J_{t_0}}(S^2,u^*\rT Q)
\cap W^{2,2}(S^2) .
    \label{eq:definition-of-F}
\end{align}
Here $\Lambda^{0,1}_J(\ldots)$ denotes continuous complex-antilinear 1-forms, and to pick out those of class $W^{2,2}$ we view them as functions on $S^2$ with values in a bundle whose fiber over $z\in S^2$ are the linear maps $\rT_z S^2 \to \rT_{u(z)} Q$. 
This Hilbert space can be equipped with a scale structure whose $k$-th level consists of forms of regularity $W^{2+k,2}$. So we can define a trivial strong bundle\footnote{As a set, $\cO \vartriangleleft F$ is the Cartesian product $\cO \times F$. The symbol $\vartriangleleft$ means that we consider it as a strong bundle, which is a condition on the scale structure: For $\eta\in\cO$ on level $m$, i.e.\ of Sobolev class $W^{3+m,2}$, it makes sense to talk about sections along $u':=\exp_u\eta$ of Sobolev class up to $W^{3+m,2}=W^{2+k,2}$, i.e.\ up to level $k=m+1$.}
\begin{align}
    I_{t_0} \times \cO \vartriangleleft F \rightarrow I_{t_0} \times \cO
\end{align}
which we will use as local model for $\cE \rightarrow [0,1] \times \cB$ near the pair $(t_0,[u])$.

To trivialize the bundle, we need to map $J_{t_0}$-antilinear one-forms with values in $u^*\rT Q$ to $J_{t}$-antilinear one-forms with values in $v^*\rT Q$ for pairs $(t,v)$ near $(t_0,u)$. We will do this in two steps, first changing the almost complex structure and then the pullback bundle.
For every $t\in I_{t_0}$ we can define a linear map from the space of antilinear 1-forms with respect to $J_0$ to the space of antilinear 1-forms with respect to $J_t$ (both with values in $u^*\rT Q$),
\begin{align} \label{eq:defining-Kt}
    K_{t} \, : \; F= \Lambda^{0,1}_{J_{t_0}}\left( S^2, u^*\rT Q \right) &\longrightarrow \Lambda^{0,1}_{J_{t}}\left( S^2, u^* \rT Q\right)  \\ \nonumber
   \xi &\longmapsto \tfrac{1}{2} \left( \xi + J_t(u) \circ \xi \circ i \right).
\end{align}

Note that $K_{t_0}$ is the identity map. 
Moreover, the explicit form of $K_t$ allows us to check that this family is smooth with respect to the $W^{2+k,2}$-topology on $\Lambda^{0,1}(\ldots)$ for any $k\in\N_0$.\footnote{For fixed $k$, the continuity of $K_t$ requires $u\in\cC^{2+k}$, which is why we center these charts at smooth points $[u]\in\cB_\infty$.} 
So by choosing $I_{t_0}$ as sufficiently small neighbourhood of $t_0$ we can guarantee that $K_t$ is a linear sc-isomorphism for all $t\in I_{t_0}$.

The rest of the bundle chart construction -- i.e.\ the step from the pullback bundle $u^*\rT Q$ to the pullback bundle $v^*\rT Q$ -- proceeds as in \cite{HWZgw}; we just need to ensure sc-smooth dependence on the extra parameter. We construct a family of connections $\tilde{\nabla}^t$ on $\rT Q$ for $t\in [0,1]$ as follows:\footnote{Any smooth family of complex connections suffices for the present construction. In \S \ref{sec:linearization}, however, we will need $\tilde{\nabla}^0$ to split along the factors of $Q=\CP^1\times T$, see Remark~\ref{rmk:parallel-transport-at-0}.} 
Let $\nabla^t$ denote the Levi--Civita connection of the metric $g_t:=\omega(\cdot,J_t\cdot)$ on $Q$ and define a new connection by $\tilde{\nabla}^t_XY := \tfrac{1}{2}(\nabla^t_XY -J_t\nabla^t_X(J_tY))$. Then the family is smooth in $t$, and $\tilde{\nabla}^t$ is a complex connection on the almost complex vector bundle $(\rT Q,J_t)$, that is, it satisfies
$\tilde{\nabla}^t_X(J_tY) = J_t(\tilde{\nabla}^t_XY)$.
Moreover, recall that the good data used to construct the above chart for $\cB$ included the choice of an open neighbourhood $\tilde{\cO}$ of the zero section of $\rT Q$, such that $\tilde{\cO}$ is fiberwise convex and the exponential map (for a fixed metric on $Q$ that does not vary with $t$) restricts to an embedding on each fiber $\tilde{\cO}_q$, $q\in Q$. Then for a tangent vector $\eta_q \in\tilde{\cO}_q$, consider the geodesic path $[0,1]\rightarrow Q, s\mapsto \exp_q(s\eta_q)$ from $q$ to $p:= \exp_q(\eta_q)$. Parallel transport with respect to $\tilde{\nabla}^t$ along this path defines a $J_t$-complex linear map\footnote{
Note here that parallel transport is defined along any path, so there is no issue with the fact that the path is induced by an exponential map for a different metric than the family of metrics used in the construction of the connection.
} 
\begin{align} \label{eq:gamma}
    \Gamma_t(\eta_q) \, :\;  (\rT_q Q, J_t(q)) \longrightarrow (\rT_pQ, J_t(p)) .
\end{align}
This in fact is an isomorphism for each $t\in[0,1]$ and $\eta_q\in\rT_q Q$, and these isomorphisms vary smoothly with $t\in[0,1]$, $q\in Q$ and $\eta\in TQ$.

The resulting bundle chart covering the chart \eqref{eq:definition-of-chart} of the base space given by $I_{t_0}\times \cO \rightarrow [0,1)\times \cB, (t,\eta)\mapsto (t, [\exp_u\eta])$ is
\begin{align}
    I_{t_0} \times \cO \vartriangleleft F &
    \; \longrightarrow \; \cE  \nonumber \\ 
    (t,\eta,\xi) \; &\longmapsto \; \Big(t, \Big[\big( \,\exp_u\eta \,,\, \Xi(\eta,\xi):= \Gamma_t(\eta)\circ K_t(\xi) \,\big)\Big] \Big) .  \label{eq:definition-bundle-chart}
\end{align}
Here the complex antilinear 1-form $\Xi(\eta,\xi) \in \Lambda^{0,1}_{J_t}(S^2,v^*\rT Q)$ with values in the pullback bundle by $v:= \exp_u \eta$ is given at each $z\in S^2$ by the complex antilinear map
\begin{align*}
    \Gamma_t(\eta(z))  \circ K_{t} (\xi)(z) 
\, :\;  \big(\rT_z S^2,i\big) \to \big(\rT_{u(z)}Q, J_t(u(z))\big) \to \big(\rT_{v(z)}Q, J_t(v(z))\big)  .
\end{align*}
It coincides with the construction in \cite[(3.9)]{HWZgw} in case $K_t=\id$ due to $J$ being fixed.
We can now follow the arguments of \cite[\S3.6]{HWZgw} to construct $\cE$ as tame strong M-polyfold bundle. For that purpose first note that each map \eqref{eq:definition-bundle-chart} covers an M-polyfold chart for $[0,1]\times\cB'\subset [0,1]\times\cB$, so the images of these maps cover only $\cE|_{[0,1]\times\cB'}$. 
Second, these are bundle charts in the sense that they are linear bijections on each fiber. They are local homeomorphisms because the topology on the total space $\cE|_{[0,1]\times\cB'}$ is defined in this manner as in \cite[Thm.1.9]{HWZgw}.
Strong sc-smooth compatibility of these bundle charts is proven for fixed $J$ in \cite[Prop.3.39]{HWZgw}, using the language of \cite[Prop.3.39]{HWZbook}. 
The strong bundle isomorphism $\mu$ that is considered here, and proven to be a local sc-diffeomorphism, in fact encodes all transition maps between different bundle charts. 
The proof of \cite[Prop.3.39]{HWZgw} directly extends to the case of varying $J$ thanks to its explicit nature \eqref{eq:defining-Kt} of $K_t$ as family of linear $0$-th order operators. Here we have to again require $u\in\cC^\infty$ to ensure that $K_t$ is a bounded operator on each of the $W^{2+k,2}$-scales. In fact, these operators vary smoothly with $t\in[0,1]$ on each scale. 
Thus the charts \eqref{eq:definition-bundle-chart} equip $\cE|_{[0,1]\times\cB'}$ with the structure of a strong M-polyfold bundle. Finally, tameness is a condition on the underlying M-polyfold that is automatically satisfied in our case since all retractions are trivial; see \cite[Def.2.17]{HWZbook}.  
\end{proof}

\begin{rmk} \label{rmk:parallel-transport-at-0} \rm 
    For $t=0$, remember that $J_0= i \oplus J_T$ splits along the factors of $Q=\CP^1\times T$. This means that also the metric splits and so does its Levi-Civita connection $\nabla$. Then also the connection $\tilde{\nabla}^0$ splits, and thus the parallel transport map $\Gamma_0$ in \eqref{eq:gamma} preserves the factors of $\rT_{(z,p)}Q = \rT_z \CP^1 \times \rT_p T$.
\end{rmk}

\begin{rmk} \label{rmk:proof a} \rm
An alternative proof of Theorem~\ref{thm:polyfold-structure-bundle} is to construct the bundle $\cE\to[0,1]\times\cB$ in \eqref{eq:definition-of-E} from an implicit function theorem in \cite{ben}.

This proof requires as starting point a polyfold description of the Gromov-Witten moduli space $\cM$ for a family $(J_t)_{t\in[0,1]}$ of almost complex structures. Such a description is obtained by performing the constructions of Theorem~\ref{thm:polyfold-structure-bundle} over $[0,1]\times Z$ to obtain a tame strong M-polyfold bundle $\widetilde p: \widetilde W \to [0,1]\times Z$ which restricts on every slice $\{t\}\times Z$ to the bundle $W\to Z$ from \cite{HWZgw} for the almost complex structure $J_t$.

Now the projection $\cE\to[0,1]\times\cB$ in \eqref{eq:definition-of-E} is obtained from $\widetilde W\to [0,1]\times Z$ by restriction to $\cB= \text{ev}^{-1}(\{p_0\})\subset Z$ as in the first proof of Theorem~\ref{thm:polyfold-structure-base-space}. 
Here the map $[0,1]\times Z \to Q, (t,[u])\mapsto \ev{[u]}=u(z_0)$ is transverse to $\{p_0\}\in Q$ since the evaluation map -- without the $[0,1]$-factor -- is already transverse by \cite[\S1.5]{ben}.
Thus \cite[Theorem~5.10 (II)]{ben} provides an open neighbourhood $\widetilde\cB\subset[0,1]\times(\cB\cap Z_1)$ of $[0,1]\times\cB_\infty$ such that $\widetilde p^{-1}(\widetilde\cB)=\cE|_{\widetilde\cB} \to \widetilde\cB$ inherits the structure of a tame strong M-polyfold bundle. 
In our special case, the shift to $Z_1$-topology is actually not needed for the reasons already stated in Theorem~\ref{thm:polyfold-structure-base-space}, and since the retracts are all trivial. 
Furthermore, $\widetilde\cB$ can be replaced by $[0,1]\times\cB''$ for an open neighbourhood $\cB''\subset\cB$ of $\cB_\infty$ that may just be somewhat smaller than $\cB'$ from Theorem~\ref{thm:polyfold-structure-bundle}.

Indeed, $\widetilde\cB$ is given by a union of charts for $[0,1]\times\cB$ centered at smooth points, which lift to strong bundle charts. Each of these charts is constructed in product form, and since $[0,1]$ is compact we can find for every $b\in\cB_\infty$ finitely many such product charts that cover $[0,1]\times\{b\}$. This implies $[0,1]\times\cU_b \subset \widetilde\cB$ for some open neighbourhood $\cU_b\subset\cB$ of $b$. This proves the claim with $\cB'=\bigcup_{b\in\cB_\infty} \cU_b$.  
\end{rmk}

\subsection{The section}
\label{sec:section}

This section finalizes the polyfold description of the moduli space $\cM=\sigma^{-1}(0)$ by establishing the relevant properties of the section $\sigma:[0,1]\times \cB\rightarrow \cE$ introduced in \eqref{eq:definition-of-sigma}. Up to quotienting by reparameterization in the base, its principal part (given by its values in the fibers) is $(t,u) \mapsto \bar{\partial}_{J_t}u$.

\begin{thm}
Let $\cB'\subset\cB$ be as in Theorem~\ref{thm:polyfold-structure-base-space}\footnote{Recall we expect $\cB'=\cB$ by Remark~\ref{rmk:smaller-base-space-no-problem}.}, then
\label{thm:sc-Fredholm}
    $\sigma:[0,1]\times \cB'\rightarrow \cE$ is a sc-Fredholm section of Fredholm index $1$.
\end{thm}

As for the bundle structure in Remark~\ref{rmk:proof a}, the Fredholm property of the section could be proven by the restriction results of Filippenko \cite{ben} -- if the Fredholm property of the Cauchy-Riemann operator with varying $J$ was firmly established. 
We explain this approach in Remark~\ref{rmk:proof a section}, after giving a direct proof of the Fredholm property based on the explicit bundle charts in Theorem~\ref{thm:polyfold-structure-bundle}.

\begin{proof}[Proof of Theorem~\ref{thm:sc-Fredholm}]
We work in local coordinates centered at a pair $(t_0,[u])\in[0,1]\times\cB_\infty$ that were defined in \eqref{eq:definition-bundle-chart}. The principal part of the section $\sigma$ is given in these coordinates by
\begin{align*}
 f: \;   I_{t_0} \times \cO \; \longrightarrow \; F , \quad
    (t,\eta) \; &\longmapsto \;  K_t^{-1}\bigl(\Gamma_t(\eta)^{-1}( \bar{\partial}_{J_t} \exp_u\eta )\bigr) ,
\end{align*}
with the family of sc-isomorphisms $K_t : F \to\Lambda^{0,1}_{J_{t}}\left( S^2, u^* \rT Q\right) $ given in \eqref{eq:defining-Kt}, and parallel transport $\Gamma_t(\eta)$ as in \eqref{eq:gamma}. Both of these are linear and explicitly given in terms of point-wise operations. Thus routine computations show that $f$ is not just sc-smooth, but for any $k\in\N_0$ restricts to a classically smooth map with respect to the $W^{k+3,2}$-norm on $\cO$ and the $W^{k,2}$-norm on $F$. (For the present proof, it suffices to check continuous differentiability.)
We can moreover see that the section $\sigma$ has classical Fredholm linearizations at any zero $(t_0,[u])\in \sigma^{-1}(0)$. Indeed, in the coordinates centered at a point with $\bar{\partial}_{J_{t_0}} u =0$ we have $\rd f(t_0,0)(T,\zeta): \R\times E_u \to F$ given by 
\begin{equation}\label{eq:linearization}
\rd f(t_0,0)(T,\zeta) \;=\; - \tfrac 12 T \cdot J_{t_0} (\partial_t J_t)|_{t=t_0} \partial_{J_{t_0}}u
\;+\; (\rD_u \bar{\partial}_{J_{t_0}})\zeta,
\end{equation}
with $E_u$ given in \eqref{eq:definition-of-E_u}. 
The first part, $T \mapsto - \tfrac 12 T \cdot J_{t_0} (\partial_t J_t)|_{t=t_0} \partial_{J_{t_0}}u$, is a bounded linear operator $\R\to F$ with respect to any $W^{\ell,2}$-norm on $F$, hence it is compact with respect to the $W^{2,2}$-norm on $F$. The second part is a restriction of the classical Cauchy-Riemann operator $\rD_u \bar{\partial}_{J_{t_0}}: X:= W^{3,2}(S^2,u^*\rT Q) \to F$. 
This classical operator is known to be Fredholm, see e.g.\ \cite[Thm.C.1.10]{McDuffSalamon2}, and restriction to the finite codimension subspace $E_u\subset X$ preserves the Fredholm property. This shows that $\rd f(t_0,0)$ is classically Fredholm with index given by the index of the restriction $\rD_u \bar{\partial}_{J_{t_0}}|_{E_u}$ plus the dimension of the domain $\R$ of the compact factor. 
The index of $\rD_u \bar{\partial}_{J_{t_0}}$ is $2n +2 c_1([\mathbb{CP}^1\times \{\textup{pt}\}])$ by the Riemann-Roch Theorem. 
The Chern number can be computed with respect to any compatible almost complex structure on $Q$, and for $J_0=J_{\CP^1}\oplus J_T$ we have
\begin{align*}
    c_1([\mathbb{CP}^1\times \{\textup{pt}\}])
    &= \int_{\mathbb{CP}^1\times \{\textup{pt}\}} c_1(\rT Q,J_0) \\
    &= \int_{\mathbb{CP}^1} c_1(\rT \CP^1,J_{\CP^1}) +\int_{\{\textup{pt}\}} c_1(\rT T,J_T) \;=\; 2 + 0 . 
\end{align*}
Finally, restricting an operator to a subspace reduces the Fredholm index (via a mix of its effects on kernel and image) by the codimension of the subspace. In this case, $E_u\subset W^{3,2}(S^2,u^*\rT Q)$ has codimension $2n+4$ since it is given by the codimension $2n$ condition $\eta(z_0)=0$ and the two codimension $2$ conditions $\eta(z_i)\in H_{p_i}$. Thus we obtain the claimed Fredholm index
$$
\text{index}(\rd f(t_0,0)) \;=\; 1 + 2n + 2c_1([\mathbb{CP}^1\times \{\textup{pt}\}]) - 2n - 4 \;=\; 1. 
$$
This also establishes the classical Fredholm property of the section $\sigma$ in local coordinates. 
The nonlinear sc-Fredholm property of polyfold theory, however, demands more than just the linearizations being sc-Fredholm.\footnote{The linear sc-Fredholm property \cite[Def.~1.8]{HWZbook} is a direct analogue of the classical linear Fredholm property -- kernel and image need to have complements that respect the sc-structure, these complements need to be sc-isomorphic, and kernel and cokernel should be finite dimensional.} The sc-Fredholm property of a section as defined in \cite[Def.~3.8]{HWZbook} requires three conditions. The first, sc-smoothness, follows from the classical smoothness in local coordinates. The second condition, $\sigma$ being regularizing, means that if $(t,v)\in [0,1]\times W^{3+m,2}$ with $\bar{\partial}_{J_t} v \in W^{2+m+1,2}$, then $(t,v)\in[0,1]\times W^{3+m+1,2}$. This follows from the corresponding property of $\partial_{J_t}$ for fixed $t$ -- which is known from elliptic regularity. The third condition seems less transparent but is very important for the implicit function theorem in scale calculus: At every smooth point $(t_0,[u])\in[0,1]\times\cB_{\infty}$, the section needs to have the sc-Fredholm germ property \cite[Def.~3.7]{HWZbook}, that is, after subtraction of a local $sc^+$-section, (a filled version of)\footnote{In our setting, there is no need for a filling since all retractions are trivial.} the germ of $\sigma$ needs to be conjugate to a basic germ as defined in \cite[Def.~3.6]{HWZbook}. The latter means that after splitting off a finite dimensional factor from the domain and projecting to the complement of a finite dimensional factor in the image, the germ is the identity plus a contraction mapping.
It is this third property that implies that linearizations of a sc-Fredholm section are sc-Fredholm operators, see \cite[Prop.~3.10]{HWZbook}. The index is then defined as the Fredholm index of the linearization at the lowest level of the scale structure, which we computed above to be $1$.

To establish the equivalence to a contraction germ normal form, we proceed similar to \cite[Prop.~4.26]{HWZgw}, using the fact that the section is classically differentiable in all but finitely many directions. In fact, the local representative $f$ above is continuously differentiable in all directions, and thus satisfies the conditions of being sc-Fredholm with respect to the trivial splitting $E_u\cong \{0\}\times E_u$, as defined in \cite[Def.~4.1]{W-Fred}.
Indeed, we already established the regularizing property (i). The differentiablity conditions (ii) in the trivial splitting follow from classical continuous differentiability. Besides, in this setting the linearized sc-Fredholm property (iii) is only required of $\rD_0 f$ -- though for any (not necessarily holomorphic) base point $(t_0,[u])$. 
We can make up for the latter complication by subtracting from $f$ the sc$^+$-section 
\begin{align*}
 f_0 : \;   I_{t_0} \times \cO \; \longrightarrow \; F , \quad
    (t,\eta) \; &\longmapsto \;  K_t^{-1}\bigl( \bar{\partial}_{J_t} u \bigr) , 
\end{align*}
which takes the same value at $(t_0,0)$ as $f$. Thus the linearization of $f-f_0$ at $(t_0,0)$ is well defined, and it can be computed as follows:
$$\rd (f-f_0)(t_0,0)(T,\zeta) \;=\; T \cdot \partial_t(\Gamma_t(\eta)^{-1})(\bar{\partial}_{J_{t_0}}u)
\;+\; \rD_0\Gamma_t^{-1}(\zeta)(\bar{\partial}_{J_{t_0}}u)
\;+\; \rD_u \bar{\partial}_{J_{t_0}}\zeta.$$ 
To check that this is a linear sc-Fredholm operator we use the conditions of \cite[Def.~3.1]{W-Fred}. (i) It is bounded on each level of the scale structure. (ii) It is regularizing by elliptic regularity for the linearized Cauchy-Riemann operator. Lastly, we have to check in (iii) the classical Fredholm property on the lowest level of the scale structure. Note that the first two summands actually are bounded with respect to the $W^{3,2}$-norm on $F$, and thus induce compact operators to $F$ with the $W^{2,2}$-topology. Thus the Fredholm property again follows from the corresponding property of the linearized Cauchy-Riemann operator. 
All in all, we have shown that the section $f-f_0$ in local coordinates satisfies all conditions of \cite[Thm.~4.5]{W-Fred}, which implies its contraction germ normal form.\footnote{Strictly speaking, the statement of this theorem shifts the scale structure. We can avoid this by observing that all the differentiability and linearized Fredholm properties of $f:\R\times E_u\to F$ persist w.r.t.\ the $W^{2,2}$-topology on $E_u$ and the $W^{1,2}$-topology on $F$.
} 
This finishes the proof. 
\end{proof}

\begin{rmk} \label{rmk:proof a section} \rm
An alternative proof of Theorem~\ref{thm:sc-Fredholm} is to combine an implicit function theorem from \cite{ben} with a polyfold description of the Gromov-Witten moduli space $\cM$ for a family $(J_t)_{t\in[0,1]}$ of almost complex structures. Such a description is obtained as follows. First, one follows Remark~\ref{rmk:proof a} to construct an M-polyfold bundle $\widetilde p: \widetilde W \to [0,1]\times Z$ which restricts on every slice $\{t\}\times Z$ to the bundle $W\to Z$ from \cite{HWZgw} for the almost complex structure $J_t$.
Then, the section $\overline\partial: [0,1]\times Z\to \widetilde W, (t,[u])\mapsto \bigl(t, \bigl[(u,\overline\partial_{J_t} u)\bigr]\bigr)$ is shown to be sc-Fredholm by following the arguments of \cite[Thm.4.6]{HWZgw}\footnote{To generalize the Fredholm analysis in \cite{HWZgw} to allow for a finite dimensional family of almost complex structures, note that this introduces an extra factor in the domain -- in our case $[0,1]$ -- along which the section is classically smooth. Since it is finite dimensional, it can also be split off when constructing the contraction germ normal form. 
}
or our proof of Theorem~\ref{thm:sc-Fredholm}. 
Given such a description, we claim that the sc-Fredholm property is preserved when we restrict from $[0,1]\times Z$ to the preimage $[0,1]\times \cB= \widetilde\ev^{-1}(\{p_0\})$ of the submanifold $\{p_0\}\subset Q$ under the evaluation map $\widetilde\ev:[0,1]\times Z \to Q, (t,[u])\mapsto \ev([u])$. 
For that purpose we can again quote the results by Filippenko: \cite[\S5.1]{ben} explains why the Cauchy-Riemann section, the evaluation map and the submanifold $\{p_0\}\subset Q$ satisfy all compatibility conditions of \cite[Theorem 5.10 (III)]{ben}. 
A direct application of that result asserts that $\sigma|_{\widetilde\cB} : \widetilde\cB \rightarrow W|_{\widetilde\cB}$ is sc-Fredholm for some open subset $\widetilde\cB\subset[0,1]\times\cB_1$ containing $[0,1]\times\cB_\infty$.
However, the shift in topology is again not needed since the evaluation map is classically smooth on all levels. 
Furthermore, as in Remark~\ref{rmk:proof a}, the open subset $\widetilde\cB\subset[0,1]\times\cB$ can be replaced by $[0,1]\times\cB'''$ for an open neighbourhood $\cB'''\subset\cB$ of $\cB_\infty$ that may just be smaller than $\cB'$ from 
Theorem~\ref{thm:polyfold-structure-bundle}.
\end{rmk}

\subsection{Linearization}
\label{sec:linearization}

The goal of \S\ref{sec:transversality-at-boundary} will be to prove transversality of the unperturbed section $\sigma$ at $t=0$. 
For that purpose we will need to consider its linearization at the unique solution $[u_0]\in\cB$ for $t=0$ (see \S\ref{sec:unique-J0-curve}). In fact, it will be sufficient to consider the linearization of $\sigma_0:=\sigma (0, \cdot) : \cB \rightarrow \cE|_{\{0\}\times\cB}$ at $[u_0]\in\cB$, which is what we will compute now.
We do the computation in a local chart for the restricted bundle $\cE|_{\{0\}\times\cB}$ centered at $[u_0]$, given by
$$
    \cO \times F \;\to\; \cE|_{\{0\}\times\cB} , \qquad
    (\eta,\xi) \;\mapsto\; \left(\, 0 \,,\, [ \exp_{u_0}\eta , \Gamma_t(\eta) \circ\xi ] \, \right) .
$$
This chart is obtained from the chart \eqref{eq:definition-bundle-chart} for the bundle $\cE \rightarrow [0,1]\times \cB$ centered at $(0,[u_0])$ as in \S\ref{sec:bundle}, by restriction to $\cO \times F \cong \{ 0 \} \times \cO \vartriangleleft F \subset [0,1] \times \cO \vartriangleleft F$. Here $\cO\subset E_{u_0}$ is an open subset of the vector space $E_{u_0}$ given in \eqref{eq:definition-of-E_u0}, $F$ is defined in \eqref{eq:definition-of-F}, the exponential map is induced by a fixed metric chosen as part of the good data in the second proof of Theorem~\ref{thm:polyfold-structure-base-space}, and $\Gamma_t$ is defined in \eqref{eq:gamma}.
The map $K_t$ from \eqref{eq:defining-Kt} does not show up here because we have $t=t_0=0$ and so $K_t$ is the identity map.
In this chart, the restricted section $\sigma_0$ from \eqref{eq:definition-of-sigma} is given by
$$
    \bar{\sigma}_0 \,:\; \cO \;\to\;  \cO \times F ,\qquad \eta \;\mapsto\; \bigl(  \eta \,, \, 
     \Gamma_0 ( \eta ) ^{-1} \circ \bar{\partial}_{J_0} \big( \exp_{u_0} \eta \big) \,\bigr).
$$
Since $u_0$ is the center of the chart, it corresponds to $\eta =0$.
So the linearized operator $\rD_{u_0} \sigma_0$ in the coordinates of this chart is represented by
$\rd \bar{\sigma}_0 (0) : E_{u} \to F$, which we can compute for $\hat{\eta}  \in E_{u}$ as follows: 
\begin{align*}
\rd \bar{\sigma}_0 (0) (\hat{\eta}) 
&\;=\; \frac{d}{d\theta}\Big\vert_{\theta =0} \xi(\theta\hat{\eta})  
\;=\;\frac{d}{d\theta}\Big\vert_{\theta =0} 
    \underbrace{\Gamma_0 \big(\theta\hat{\eta} \big)^{-1}}_{=:A_{\theta}} 
    \circ   
    \underbrace{\bar{\partial}_{J_0} \big( \exp_{u_0}\theta\hat{\eta} \big) }_{=:\mu_{\theta}} \\
&\;=\; \underbrace{A_{0}}_{=\id} \circ \frac{d}{d\theta}\Big\vert_{\theta =0} \mu_{\theta} + \frac{d}{d\theta}\Big\vert_{\theta =0} A_{\theta}\circ \underbrace{\mu_{0}}_{= 0}
\;=\; \tfrac{1}{2}\left(\nabla\hat{\eta}+J_0(u_0)\circ\nabla\hat{\eta}\circ i\right). 
\end{align*}
Here $\nabla$ denotes the Levi-Civita connection on $Q$ corresponding to the metric $g_0$ that is compatible with $J_0$.
It is $A_0 = \Gamma_t \left( 0 \right)^{-1} =\id_{u_0^*\rT Q}$, as for every $z\in S^2$ it is (the inverse of) parallel transport from $u_0(z)$ to itself via the constant path,
and $\mu_{0} = \bar{\partial}_{J_0} \left( \exp_{u_0} 0 \right) = \bar{\partial}_{J_0} u_0=0$ since $u_0$ is $J_0$-holomorphic. In the last step we used a formula from \cite[Prop. 3.1.1]{McDuffSalamon2} and again the fact that $u_0$ is $J_0$-holomorphic.
Thus we have shown that the polyfold-theoretic linearization $\rD_{u_0}\sigma_0$ in an appropriate bundle chart is given by a restriction of the classical Cauchy-Riemann operator of the complex bundle $(u_0^*\rT Q, J_0(u_0))$,
$$
\rd \bar{\sigma}_0 (0) : E_{u} \;\to\; F, \qquad \hat\eta \;\mapsto\; \tfrac{1}{2}\left(\nabla\hat{\eta}+J_0(u_0)\circ\nabla\hat{\eta}\circ i\right) \;=\;  \rD_{J_0(u_0)} \hat\eta  . 
$$
Indeed, this operator differs from the classical Cauchy-Riemann operator $\rD_{J_0(u_0)}: X \to F$ as in \cite[Rmk.C.1.2]{McDuffSalamon2} only by the domain $E_{u_0}$ being a subspace of $X:= W^{3,2} ( S^2 , u_0^* \rT Q)$.

\subsection{Transversality at the boundary}
\label{sec:transversality-at-boundary}

The last missing ingredient for the proof of Gromov's nonsqueezing Theorem~\ref{thm-intro:J1-sphere} in \S\ref{sec:applying-regularization-scheme}
is to show that the unique $J_0$-holomorphic curve is cut out transversely. 

\begin{rmk}\rm
The following proof is a first instance of the general principle ``classical transversality implies polyfold transversality''. The core difference between the two notions is that, classically, one usually proves surjectivity of a linearized Cauchy-Riemann operator $\rD: X\to F$ on a tangent space $X=\rT_{u_0}\mathcal{X}$ to a space $\mathcal{X}$ of all maps (in a given homology class, etc.). In our case (ignoring the homology class and point constraint $u(z_0)=p_0$), this total space would be $\mathcal{X}= W^{3,2} ( S^2 , Q)$). This space still carries the action of a group of reparameterizations. In our case a group $\text{Aut}$ of automorphisms of $S^2$ acts on $\mathcal{X}$, and preserves the Cauchy-Riemann operator, so that the tangent space of its orbit lies in the kernel, $\rT_{u_0} \{ u_0\circ\phi \,|\, \phi\in\text{Aut} \}\subset \ker\; \rD$. 

In contrast, the polyfold setup works with the quotient space $\cB=\mathcal{X}/\text{Aut}$ whose tangent space $\rT_{[u_0]}\cB$ is represented by a subspace $E_{u_0} \subset X= W^{3,2} ( S^2 , u_0^* \rT Q)$. 
So, the main challenge in deducing surjectivity of the linearized polyfold section $\rD_{[u_0]}\sigma = \rD|_{E_{u_0}}$ from surjectivity of the classical Cauchy-Riemann operator $\rD$ is in showing that this quotient construction results in a splitting of the total space $X = E_{u_0} + A$ with a complement $A\subset\ker\; \rD$ that represents the infinitesimal action of reparameterizations. 
\end{rmk}

\begin{thm}
    \label{thm:transverslity-at-0}
    $\sigma$ is transverse to the zero section at $t=0$.
\end{thm}

\begin{proof}
Since $u_0$ is the only solution for $t=0$ (see Lemma~\ref{lem:only-one-J0-curve}), for checking transversality of $\sigma$ at the boundary $t=0$ it suffices to consider the linearization of $\sigma$ at $(0,u_0)$ and show that it is surjective.
For this it is sufficient to show that the linearization of $\sigma_0:=\sigma(0,\cdot)$ is surjective.
In \S\ref{sec:linearization} we computed this linearization $\rD_{u_0} \sigma_0$ in a local chart centered at $u_0\in\cB$ to be the restriction
$\rd\overline{\sigma}_0(0) = \rD_{J_0(u_0)}|_{E_{u_0}}: E_{u_0}\to F$ of the classical Cauchy-Riemann operator $\rD_{J_0(u_0)}: X \to F$ to the subspace $E_{u_0} \subset X= W^{3,2} ( S^2 , u_0^* \rT Q)$ given by \eqref{eq:definition-of-E_u0}. Its codomain $F$ is defined in  \eqref{eq:definition-of-F}. 
We will first show that this classical operator is surjective; see also \cite[Lemma~5.5]{Wendl}. 

Recall from \S\ref{sec:unique-J0-curve} that $u_0: S^2 \to Q=\CP^1\times T, z \mapsto (z,m_0)$ is the product of the identifying map $S^2\cong\CP^1$ and a constant map to the torus. Thus the complex structure along $u_0$ splits $J_0(u_0)= J_{\CP^1}\oplus J_T(m_0)$ into the standard complex structure $J_{\CP^1}$ on $\rT \CP^1 \cong \id^*\rT\CP^1$ and the constant complex structure $J_T(m_0)= J_{\text{st}}$ on $\rT_{m_0}T=\R^{2n-2}$. 
Thus we have a natural splitting of complex vector bundles
\begin{align*}
   u_0^* (\rT Q , J_0) \;=\; ( \rT \CP^1 , J_{\CP^1}) \oplus E_0^{(n-1)} , 
\end{align*}
where $E_0^{(n-1)}$ is the trivial complex bundle of rank $n-1$ over $S^2$ with fibers $(\R^{2n-2}, J_{\text{st}})$. 
The domain of the Cauchy-Riemann operator thus splits into
\begin{align*}
    X \;=\; W^{3,2} \left( S^2 , u_0^* \rT \left( \CP^1 \times T \right) \right)
    \;=\; \underbrace{W^{3,2} (S^2 , \rT \CP^1 )}_{=:X_{\CP^1}} \times \underbrace{W^{3,2} (S^2 , E_0^{(n-1)} )}_{=: X_T}, 
\end{align*}
and, analogously, the codomain splits $F=F_{\CP^1}\times F_T$ into the $W^{2,2}$-closures of smooth complex antilinear 1-forms on $S^2$ with values in $\rT\CP^1$ resp.\ $E_0^{(n-1)}$. This shows that the classical Cauchy-Riemann operator splits
$$
\rD_{J_0(u_0)} =  \rD_{J_{\CP^1}} \oplus \rD_{J_{\text{st}}} 
\,:\; X = X_{\CP^1}\times X_T \;\to\; F=
F_{\CP^1}\times F_T 
$$
into the classical Cauchy-Riemann operator $\rD_{J_{\CP^1}}:X_{\CP^1}\to F_{\CP^1}$ of the complex line bundle $( \rT \CP^1 , J_{\CP^1})$ over $\CP^1\cong S^2$ and the Cauchy-Riemann operator $\rD_{J_{\text{st}}}:X_T\to F_T$ of the trivial bundle $E_0^{(n-1)}$ over $S^2$. The latter splits further $\rD_{J_{\text{st}}}=\rD_i \oplus \ldots\oplus \rD_i$ into $n-1$ copies of the Cauchy-Riemann operator $\rD_i$ of the trivial complex line bundle $E_0^{(1)} = \C\times S^2\to S^2$. 
These complex line bundles satisfy $c_1(\rT\CP^1)+2\chi(S^2) =2 + 2\cdot 2 >0$ resp.\ $c_1(E_0^{(1)})+2\chi(S^2) = 0 + 2\cdot 2 >0$, so the Riemann-Roch Theorem 
(e.g. \cite[Thm.C.1.10.(iii)]{McDuffSalamon2}) guarantees that the corresponding Cauchy-Riemann operators are surjective. 
This proves surjectivity of the product $\rD_{J_0(u_0)}$ of these surjective operators arising from complex line bundles. 

Towards proving surjectivity of the restriction $\rd\overline{\sigma}_0(0) = \rD_{J_0(u_0)}|_{E_{u_0}}: E_{u_0}\to F$, 
note the above surjectivity means $\rD_{J_0(u_0)} (X) = F$.
If we can now show that
$X=E_{u_0} + V$ can be written as the sum of $E_{u_0}$ and a subspace of the kernel $V\subset \ker\, \rD_{J_0(u_0)}$, then it follows that $\rD_{J_0(u_0)} (E_{u_0}) = \rD_{J_0(u_0)}(E_{u_0}+ V) =F$, so the restriction $\rd\overline{\sigma}_0(0) = \rD_{J_0(u_0)}|_{E_{u_0}}$ is surjective as well.
To find this subspace $V$ note that our choice of transverse hypersurfaces in Remark~\ref{rmk:good-data-at-u_0} was made such that the vector space $E_{u_0}$ defined in  \eqref{eq:definition-of-E_u0} splits as
$E_{u_0} = E_{\CP^1} \times E_{T}$, 
where
$$
    E_{\CP^1} \;=\; \lbrace \eta \in X_{\CP^1} \mid \eta (z_i) = 0 \text{ for } i= 0,1,2  \rbrace , \qquad
    E_{T} \;=\; \lbrace \eta \in X_T \mid \eta (z_0) = 0  \rbrace . 
$$
So we can construct $V\subset X$ as product $V := V_{\CP^1} \times V_T$ of subspaces satisfying
\begin{itemize}
    \item[(a)] $X_{\CP^1} = E_{\CP^1} + V_{\CP^1}$ and $ \rD_{J_{\CP^1}} ( V_{\CP^1} ) = 0$,
    \item[(b)] $X_T = E_{T} + V_T$ and $\rD_{J_{\text{st}}} ( V_T ) = 0 $.
\end{itemize}
To meet these requirements, we choose the subspaces of holomorphic sections 
$$
    V_{\CP^1} \,:=\; \rT_{\mathrm{id}} \mathrm{Aut} (S^2) \quad \text{and}\quad 
    V_T \,:=\; \{ \eta_T : S^2 \to \rT_{m_0} T \;\text{ constant} \} .
$$
With this choice it is easy to verify (b): Every $\eta \in X_T$ is the sum of the constant vector field taking the value $\eta (z_0)$ and the vector field $\eta (\cdot) - \eta(z_0) \in E_{T}$. Moreover, $\rD_{J_{\text{st}}} ( V_T ) = 0 $
holds since all the elements of $V_T$ are constant.\\
To verify (a), first recall that ${\rm Aut}(S^2)$ is the group of biholomorphisms $S^2\to\CP^1$, i.e.\ solutions $\psi:S^2\to\CP^1$ of the nonlinear Cauchy-Riemann equation $\overline\partial_{J_{\CP^1}}\psi=0$. So the tangent space $\rT_{\mathrm{id}} \mathrm{Aut} (S^2)$ to this finite dimensional family of holomorphic maps at the identity map is a subspace of the kernel of the linearized Cauchy-Riemann operator $\ker\,\rD_{J_{\CP^1}}\subset X_{\CP^1}$. In particular, $V_{\CP^1}=\rT_{\mathrm{id}} \mathrm{Aut} (S^2)$ is a subspace of the space of smooth sections $\cC^\infty(S^2,\rT \CP^1) \subset X_{\CP^1}=W^{3,2} (S^2 ,\rT \CP^1 )$ since these tangent vectors are obtained as derivatives of paths of maps in $\mathrm{Aut} (S^2)$, 
\begin{align*}
    \rT_{\mathrm{id}} \mathrm{Aut} (S^2) =
     \bigl\{ \tfrac{\rd}{\rd t} \big|_{t= 0} \gamma(t)\,
     \big| \;\gamma \colon (-\epsilon, \epsilon) \to \mathrm{Aut}(S^2) \text{ with } \gamma(0) = \mathrm{id} \bigr\} . 
\end{align*}
These derivatives take values $\tfrac{\rd}{\rd t} \big|_{t= 0} \gamma(t) (z) \in \rT_{\gamma(0)(z)}\CP^1 = \rT_z \CP^1$ at $z\in S^2$. So, strictly speaking, they are sections of the pullback bundle $\text{id}^*\rT \CP^1$ under the identification  $\text{id}: S^2\overset{\cong}{\to}\CP^1$.  
More concretely, the automorphisms of $S^2 =\mathbb{C}\cup\{\infty\}$ are of the form $\infty \neq z \mapsto \frac{az +b}{cz+d}$ and $\infty\mapsto \frac ac$
with $a,b,c,d \in \C$ such that $ad-bc \neq 0$. Thus, each $\gamma \colon (-\epsilon, \epsilon) \to \mathrm{Aut}(S^2)$ is given by differentiable functions $a,b,c,d: (-\epsilon, \epsilon)\to\C$ that satisfy $a(0)=1=d(0)$, $b(0)=0=c(0)$ and $a(t)d(t)-b(t)c(t) \neq 0$ for all $t$.
We can compute their derivative at any $z\in S^2\setminus\{\infty\}$,
\begin{align*}
    \frac{\rd}{\rd t} \bigg\vert_{t= 0} \gamma(t)(z)
    \;=\; \frac{\rd}{\rd t} \bigg\vert_{t= 0} \frac{a(t)z +b(t)}{c(t)z+d(t)} 
     & \;=\; \frac{\dot{a}(0) z + \dot{b}(0)}{0 z + 1}  -  \frac{(1z + 0)(\dot{c}(0) z + \dot{d}(0))}{( 0 z + 1 )^2}
     \\
    &\;=\; z \dot{a}(0) + \dot{b}(0) - z^2 \dot{c}(0) - z \dot{d}(0).  
\end{align*}
Conversely, any choice of numbers $A,B,C,D \in \C$ induces a tuple of functions  $(-\epsilon, \epsilon)\to\C$ as above, given by $a(t)=1+tA$,  $b(t)=tB$, $c(t)=tC$, $d(t)=1+tD$. These define a section\footnote{Since $a(t)d(t)-b(t)c(t)\neq 1$ is an open condition and satisfied at $t=0$, by continuity it is satisfied for all small enough $t$. 
}
$\eta_V =\frac{\rd}{\rd t} \big|_{t= 0} \gamma(t) \colon S^2 \to \rT S^2$ in $X_{\CP^1}$, which on $S^2\setminus\{\infty\}\simeq \C$ is of the form\footnote{This computation uses the chart $\CP^1\setminus\{\infty\}\cong\C$. To compute the value of $\eta_V$ at $\infty\in\CP^1$, we would have to consider a chart near $\infty$. This is not needed here if we just avoid putting a marked point at $\infty$.}
\begin{equation}
    \eta_V(z) \;=\; z ( A-D) + B - z^2 C \;\in\; \C \cong \rT_z\CP^1 .
    \label{eq:description-of-A_CP1}
\end{equation}
Now, to prove (a), given any $\eta \in X_{\CP^1}$, we need to find $\eta_E \in E_{\CP^1}$ and $\eta_{V}\in V_{\CP^1}$ with $\eta =\eta_E +\eta_{V}$.
Since $\eta_E$ needs to satisfy $\eta_E (z_i)=0$ for $i=0,1,2$, we need $\eta_V$ to agree with $\eta$ on all marked points.
Without loss of generality, we can assume that $z_0 = 0$, and by Remark~\ref{rmk:good-data-at-u_0} we are free to choose the points $z_1,z_2\in S^2$ as we like. We choose $z_1 := 1$ and $z_2 := -1$. Then, given $\eta \in X_{\CP^1}$ the requirements $\eta_V(z_i)=\eta(z_i)$ translate by \eqref{eq:description-of-A_CP1} into 
\begin{align*}
B  &\;=\;\eta_V (0)  \;=\;  \eta(0),\\
A + B-  C - D &\;=\;\eta_V (1)  \;=\;  \eta(1), \\
- A + B -  C + D &\;=\;\eta_V (-1) \;=\;  \eta(-1).
\end{align*}
This system of $3$ equations for $4$ variables can be solved by choosing $D=0$, 
$$
  A \;=\;\tfrac{\eta (1)-\eta (-1)}{2},\qquad
  B\;=\;\eta (0),\qquad
  C\;=\;\eta (0)-\tfrac{\eta (1)+\eta (-1)}{2} , \qquad D=0 .
$$
As explained above, this choice defines a vector field $\eta_V\in \rT_{\mathrm{id}} \mathrm{Aut}(S^2) = V_{\CP^1}$ on $S^2$, and using \eqref{eq:description-of-A_CP1} we ensured that it has the desired values $\eta_V(z_i)=\eta(z_i)$.
Now $\eta_E:=\eta-\eta_V$ satisfies $\eta_E (z_i)=0$ for $i=0,1,2$, and hence we have $\eta =\eta_E +\eta_V$ with $\eta_E \in E_{\CP^1}$.
This finishes the proof of (a) and thus proves this theorem.
\end{proof}

\begin{rmk} \rm 
    \label{rmk:transversality-at-1}
    If $\cM_1=\emptyset$, then $\sigma$ is transverse to the zero section at $t=1$. The statement that the linearization is surjective for all $u\in \cM_1$ is then vacuously true.
\end{rmk}

\begin{appendix}
\section{The Monotonicity Lemma for pseudoholomorphic maps}
\label{app:monotonicity-lemma-for-pseudoholomorphic-maps}

The purpose of this appendix is to give a detailed proof of the monotonicity lemma for $J_{\text{st}}$-holomorphic maps to $\R^{2n}$ that was used in §\ref{sec:using-monotonicity}, which avoids the use of special properties of holomorphic maps such as their local representation. Other proofs can be found in the literature; e.g.\ \cite[Thm.5.2.1]{Wendl}. 
We also use the opportunity to establish the result in maximal generality -- for maps of regularity $\cC^1$, and with $(\R^{2n},J_{\text{st}})$ generalized to a complex Hilbert space as follows.

\begin{defn}
A {\bf complex Hilbert space} $(V,J)$ consists of a Hilbert space $V$ with inner product $\inner\cdot\cdot$ and a compatible complex structure $J$, i.e.\ an endomorphism $J:V\to V$ with $J^2 = -\id_V$ that preserves the inner product. The associated symplectic structure $\omega:V\times V \to \R$ is $\omega(v_1,v_2)=\inner{J v_1}{v_2}$.

\noindent A {\bf pseudoholomorphic map} $v:(S,j)\to(V,J)$ consists of a compact Riemann surface $(S,j)$ with (possibly empty) boundary $\partial S$ and a $\cC^1$-map\footnote{
If $V$ is finite dimensional, then $v$ is automatically smooth by elliptic regularity.
} $v:S\to V$ satisfying the Cauchy-Riemann equation $\rd v \circ j = J \circ \rd v$.
\end{defn}

\begin{lem}\label{lem:holmonotone}
Consider a nonconstant\footnote{More precisely, we assume that $v$ is not constant on any connected component of the domain. This excludes the pathological case of $v\equiv p$ on one connected component and on the other components covering just a small symplectic area in $V\setminus B_R(p)$.} $(j,J)$-holomorphic map $v:S \to V$ and an open ball $\mathring{B}_R(p):=\{q\in V \,|\, \|q-p\| < R\}$ centered at a point $p\in v(S)$ in the image, of radius $R>0$ such that $\|v(z)-p\|\geq R$ for all $z\in\partial S$. 
Then the symplectic area of $v$ within the ball is at least the area of the flat disk of radius $R$, that is 
$$ \int_{v^{-1}(\mathring{B}_R(p))} v^*\omega \ge \pi R^2.$$
\end{lem}
\begin{proof}
We begin by rewriting the 2-form $v^*\omega$ on $S$ in local holomorphic coordinates $s+it\in\C$ for $(S,j)\simeq(\C,i)$ as
\begin{align}
v^*\omega 
&\;=\; 
\omega(\partial_s v,\partial_t v) \, \rd s \wedge \rd t \nonumber\\
&\;=\; 
\tfrac 12  \bigl(\|\partial_s v\|^2 + \|\partial_t v\|^2 \bigr) \, \rd s \wedge \rd t  \label{eq:vstaromega} \\
&\;=\;
\tfrac 12 \left\langle\bigl( \partial_s v \, \rd s + \partial_t v \, \rd t \bigr) \wedge \bigl( \partial_s v \, \rd t - \partial_t v \, \rd s \bigr) \right\rangle
\;=\;
\tfrac 12 \left\langle \rd v \wedge * \rd v \right\rangle . \nonumber
\end{align}
Here the Cauchy-Riemann equation $\partial_s v + J\partial_t v = 0$ in local coordinates together with compatibility of $\omega$ and $ J$ with the inner product implies
$$
\omega (\partial_s v , \partial_t v ) 
= \langle J \partial_s v , \partial_t v \rangle 
= \langle - J^2 \partial_t v , \partial_t v \rangle
= \|\partial_t v\|^2 = \|\partial_s v\|^2 .
$$ 
The final result $\tfrac 12 \left\langle \rd v \wedge * \rd v \right\rangle$ is a well defined global expression (i.e.\ independent of coordinates) that uses the Hodge operator $*:\rT S \to \rT S$ induced by the complex structure $j$. In local holomorphic coordinates the Hodge operator is given by $*\rd s = \rd t$ and $*\rd t = -\rd s$. 
Here and below the notation $\langle \alpha \wedge \beta \rangle$ for differential forms with values in $V$ denotes the wedge product $\wedge$ on the level of differential forms, with two values in $V$ being multiplied via the inner product $\langle \cdot, \cdot \rangle$.

The second expression for $v^*\omega$ in \eqref{eq:vstaromega} shows that the area $A(r):= \int_{v^{-1}(\mathring{B}_r(p))} v^*\omega$ is non-negative and grows monotone with $r$, since it integrates a non-negative multiple of the area form on $S$ over domains $v^{-1}(\mathring{B}_r(p))$ that grow with $r$. Classical monotonicity proofs now argue that the ratio $a(r):=r^{-2} A(r)$ as function\footnote{
Monotonicity of $A(r)$ implies that this function is differentiable almost everywhere and has at most countably many jump discontinuities. Then the same holds for the ratio function $a(r)$ since $r\mapsto r^{-2}$ is smooth on the domain $(0,R]$. 
} 
of $r\in (0,R]$ satisfies $\frac{\rd}{\rd r}a\ge 0$ and $\lim_{r\to 0} a(r) \ge \pi$, which implies the claim $a(R)\ge \pi$.
We will follow the same line of argument but avoid differentiability concerns by establishing a uniform difference estimate
\begin{equation}\label{eq:monotone}
A(r+\eps) - A(r)  \;\geq\; \tfrac{2 \eps}{r+\eps} A(r) 
\qquad\textup{ for all }\; 0<r<R,\; 0<\eps<R-r.
\end{equation}

\smallskip\noindent
{\bf Proof of difference estimate:}
To prove \eqref{eq:monotone}, we simplify notation by assuming without loss of generality that $p=0$. This can be achieved by applying a global shift which does not affect the area. We can estimate the area of $v$ in $\mathring{B}_r(0)$ by 
\begin{equation}\label{eq:upper}
\textstyle
A(r) =  \int_{v^{-1}(\mathring{B}_r(0))} v^*\omega
\;\leq\;
\tfrac 12 \int_S  f_\eps(\|v\|) \langle \rd v \wedge * \rd v \rangle  , 
\end{equation}
where $f_\eps:[0,\infty)\to[0,1]$ is the continuous, piecewise linear cutoff function with $f_\eps|_{[0,r]}\equiv 1$, $f_\eps|_{[r+\eps,\infty]}\equiv 0$, and $f'_\eps|_{(r,r-\eps)}=-\eps^{-1}$.
On the other hand, integration by parts yields 
\begin{align}
\textstyle
\int_S  f_\eps(\|v\|) \langle \rd v \wedge * \rd v \rangle 
&= \textstyle\int_{\partial S}  f_\eps(\|v\|) \langle v , * \rd v \rangle 
- \int_S  f'_\eps(\|v\|) \, \rd \| v \| \wedge  \langle v , * \rd v \rangle  \nonumber\\
&\leq \textstyle  - \int_S f'_\eps(\|v\|)  \tfrac 12 \|v\|   \langle \rd v \wedge * \rd v \rangle \nonumber\\
&\leq \textstyle \frac{r+\eps}{\eps} \int_{v^{-1}(\mathring{B}_{r+\eps}(0))\setminus v^{-1}(\mathring{B}_r(0))}  \tfrac 12   \langle \rd v \wedge * \rd v \rangle  \nonumber\\
&= \textstyle \tfrac{r+\eps}{\eps} \bigl( A(r+\eps) - A(r)\bigr) . 
\label{eq:lower}
\end{align}
Here the first step uses a weak version of the Laplace equation $\rd*\rd v=0$ which follows from the Cauchy-Riemann equation $\overline\partial_J v = 0$. If $v$ is twice differentiable then this can be checked in local coordinates,
$$
* \rd*\rd v = - \partial_s^2 v - \partial_t^2 v =  ( - \partial_s + J \partial_t)  ( \partial_s v + J \partial_t v) = 
(\overline\partial_J )^* \overline\partial_J v = 0.
$$
Otherwise we first consider smooth functions $w:S\to V$ with $w|_{\partial S}=0$ and calculate using integration by parts
$$
\int_S \langle \rd w \wedge * \rd v \rangle  
= \int_S \langle \rd * \rd w , v \rangle
= \int_S \langle (\overline\partial_J )^* \overline\partial_J w , v \rangle\; {\rm dvol}
= \int_S \langle \overline\partial_J w , \overline\partial_J  v \rangle\; {\rm dvol}  = 0 .
$$
Then we note that this identity extends by continuity to $\cC^1$ functions such as $w = f_\eps(\|v\|) v$,
which vanishes on $\partial S$ since $\|v(\partial S)\|$ takes values in $[R,\infty)$, where $f_\eps$ vanishes since $r+\eps<R$. This yields the first step in \eqref{eq:lower}, with the integral over $\partial S$ vanishing. 

The second step in \eqref{eq:lower} can be checked in local holomorphic coordinates and using \eqref{eq:vstaromega} again,
\begin{align*}
 \rd \| v \|  \wedge  \langle v , * \rd v \rangle 
&=  \tfrac 1{2\|v\|} \bigl( \inner{v}{\partial_s v}\rd s + \inner{v}{\partial_t v}\rd t \bigr) \wedge   
   \bigl( \inner{v}{\partial_s v}\rd t - \inner{v}{\partial_t v}\rd s \bigr)   \\
&=  \tfrac 1{2\|v\|}  \bigl( \inner{v}{\partial_s v}^2  + \inner{v}{\partial_t v}^2 \bigr) \rd s \wedge  \rd t  \\
&\geq \tfrac {1}{2\|v\|} 
 \bigl( \|v\|^2\|\partial_s v\|^2  + \|v\|^2\|\partial_t v\|^2   \bigr)  \rd s \wedge  \rd t 
  =  \tfrac {\|v\|}{2}  \langle \rd v \wedge * \rd v \rangle .
\end{align*}
The third step in \eqref{eq:lower}  follows from $f'_\eps(\|v\|)\equiv 0$ unless $r+\eps \geq \|v\| \geq r$, and $f'_\eps = \eps^{-1}$ where it doesn't vanish.
Now combining \eqref{eq:upper} and \eqref{eq:lower} proves \eqref{eq:monotone},
$$ \textstyle
 \tfrac{r+\eps}{\eps}\bigl( A(r+\eps) - A(r) \bigr)
\;\geq\;
 \int_S  f_\eps(\|v\|) \langle \rd v \wedge * \rd v \rangle
\;\geq\; 
2 A(r)   .
$$

\smallskip\noindent
{\bf Monotone growth of ratio function:}
In terms of the ratio function $a(r)=r^{-2}A(r)$, the difference estimate \eqref{eq:monotone} implies 
\begin{align*}
(r+\eps)^2 \bigl( a(r+\eps) - a(r)\bigr) 
&= A(r+\eps) - \tfrac{(r+\eps)^2}{r^2} A(r) 
\\
&\geq A(r) + \tfrac{2 \eps}{r+\eps} A(r) - \tfrac{(r+\eps)^2}{r^2} A(r) \\
&=  \tfrac{r^3 +\eps r^2 + 2 \eps r^2 - r^3 - 3\eps r^2 - 3\eps^2 r - \eps^3}{(r+\eps)r^2} A(r) 
\; =\;   - \eps^2 \tfrac{3 r + \eps}{r+\eps}  a(r) . 
\end{align*}
Now for fixed $0<r_0<r_1<R$ we have $a(r+\eps)-a(r)\geq - C \eps^2$ for all $r\in[r_0,r_1]$ and $0<\eps<R-r_1$, with a non-negative constant 
$$
C\,:=\; \max_{r\in[r_0,r_1], \eps\in(0,R-r_1)} \frac{3 r + \eps}{(r+\eps)^3} a(r)
\;\leq\; \frac{3r_1 + (R-r_1)}{{r_0}^5} \int_S u^*\omega .
$$
Then summation with $\eps=\frac {r_1-r_0} N$ (with $N$ sufficiently large for $\eps<R-r_1$) yields 
$$
a(r_1)-a(r_0)= {\textstyle \sum_{n=0}^{N-1}  a(r_0 + n\eps +\eps ) - a(r_0 +n\eps) } \geq - N C  \bigl( \tfrac {r_1-r_0}N\bigr)^2 .
$$
Taking $N\to\infty$ this implies $a(r_1)\geq a(r_0)$ for any $0<r_0<r_1<R$.
Next, taking the limit $r_1\to R$ for fixed $r_0>0$ yields
\begin{equation}\label{eq:A(R)bound}
A(R)
\;\geq\;\lim_{r_1\to R} A(r_1) 
\;\geq\; \lim_{r_1\to R} r_1^2 a(r_1)
\;\geq\; R^2 a(r_0) .
\end{equation}
So to prove the claim $A(R)\geq \pi R^2$ it remains to establish $\lim_{r_0\to0} a(r_0)\geq\pi$.

\smallskip\noindent
{\bf Centering the ball at a regular point:}
Before studying the $r\to 0$ limit, we claim that it suffices to prove the area bound $A(R)\geq \pi R^2$ after replacing $p=v(z)$ with a sequence $p_n=v(z_n)\to p$ of images of regular points $S\setminus\partial S \ni z_n\to z$, regular just meaning that they are not critical points of $v$.
Indeed, given such a sequence and assuming the area bound holds on all balls $\mathring{B}_R(p_n)$, we have 
$$
A(R)\;=\;\int_{v^{-1}(\mathring{B}_R(p))} v^*\omega \;\geq\; \int_{v^{-1}(\mathring{B}_{R - \|p_n-p\|}(p_n))} v^*\omega \;\geq\; \pi (R - \|p_n-p\|)^2
$$ 
since $\mathring{B}_{R - \|p_n-p\|}(p_n)\subset \mathring{B}_R(p)$. This proves $A(R)\geq \pi R^2$ in the limit $\|p_n-p\|\to 0$. Moreover, such a sequence of regular points exists since the critical points of $v$ are isolated in $S$. 
For $\dim V<\infty$ this is proven in \cite[Lemma~2.4.1]{McDuffSalamon2}. 
If $\dim V=\infty$ first note that by the Cauchy-Riemann equation in local coordinates, $J\partial_s v = \partial_t v$, any $z\in S$ is either regular (i.e.\ $\rd_z v$ is injective) or critical (i.e.\ $\rd_z v=0$). 
Next, choose a complex splitting $(V,J)\simeq (\C,i)\oplus (V',J')$ in which the first component ${\rm pr}_{\C}\circ v$ is nonconstant (e.g.\ by splitting off $\im\,\rd_z v\cong\C$ at a regular point).
Then classical complex analysis asserts that the critical points of the nonconstant holomorphic map ${\rm pr}_{\C}\circ v:S\to\C$ are isolated, so in particular the critical points of $v$ are isolated.

\smallskip\noindent
{\bf Lower bound for ratio as $\mathbf{r\to 0}$:}
Once $p=v(z_0)$ is the image of a regular point $z_0\in S\setminus\partial S$ of $v$, we choose a neighbourhood $U_0\subset S$ of $z_0$ with local holomorphic coordinates $U_0\cong D_\delta:=\{(s,t)\in\R^2\,|\, s^2+t^2 <\delta^2 \}$ for some $\delta>0$, so that $z_0\cong(0,0)$ and $v|_{U_0}$ is given by $v_0(s,t)= p + s X_0 + t J X_0  + h(s,t)$ with a vector $X_0=\partial_s v_0(0,0)\in V$ of length $\|X_0\|=1$ and an error term $h$ with $h(0,0)=0$ and $\rd h(0,0)=0$. 
Then we can bound the area $A(r)\geq \int_{D_\delta} \chi\bigl(\tfrac{\|v_0-p\|}{r}\bigr) v_0^* \omega$ by an integral of the characteristic function $\chi$ with $\chi|_{[0,1]}\equiv 1$, $\chi|_{(1,\infty)}\equiv 0$. Moreover, the limit $\lim_{r\to 0}a(r)$ exists, since we already proved monotonicity of $a$ and it is bounded below by $0$. So to prove $\lim_{r\to 0}a(r)\geq\pi$ it suffices to find a sequence $r_n\to 0$ with $\lim_{n\to\infty} r_n^{-2}\int_{D_{r_n}} \chi\bigl(\tfrac{\|v_0-p\|}{r_n}\bigr) v_0^* \omega = \pi$. 
To construct this sequence of radii we use the continuous differentiability of $h$ to choose $0<r_n\leq\delta$ for each $n\in\N$ so that $\|\rd h(s,t)\|\leq \frac 1n$ and $\|h(s,t)\|\leq  \frac 1n |(s,t)|$ for $|(s,t)|\leq r_n$. 
Then on $D_{r_n }$ we can estimate
\begin{align*}
\bigl| \omega(\partial_s v , \partial_t v ) - 1  \bigr|
&\;=\;
\bigl| \omega\bigl( X_0 + \partial_s h \,,\, J X_0  + \partial_t h \bigr)  -  \omega( X_0 , J X_0  ) \bigr| \\
&\;=\;
\bigl| \omega( X_0 , \partial_t h ) + \omega( \partial_s h , X_0 ) + \omega( \partial_s h , \partial_t h ) \bigr| \\
&\;\leq\;   \tfrac 1n  + \tfrac 1n + (\tfrac 1n)^2 \; \leq \; \tfrac 3n .
\end{align*}
To control the distance $\|v(s,t)-p\|$ we use $2xy \leq  \frac 1n x^2 + n y^2$ to obtain
\begin{align*}
& \bigl| \|v(s,t)-p\|^2 - |(s,t)|^2 \bigr| 
\;=\; \bigl\| s X_0 + t J X_0  + h(s,t)\bigr\|^2 - (s^2 + t^2)   \\
&\quad\leq 2 |s| |\langle X_0 , h (s,t) \rangle | + 2 |t| |\langle JX_0 , h (s,t) \rangle | + \|h(s,t)\|^2  \\
&\quad\leq \tfrac 1n (s^2 + t^2) + n \bigl( |\langle X_0 , h (s,t) \rangle |^2 +  |\langle JX_0 , h (s,t) \rangle |^2 \bigr)   +  \|h(s,t)\|^2 \;\leq\;  \tfrac 4n |(s,t)|^2 .
\end{align*}
This holds for $(s,t)\in D_{r_n }$ and implies $\chi\bigl(\tfrac{\|v(s,t)-p\|}{r_n}\bigr)=1$ for $|(s,t)|\leq \frac{r_n}{\sqrt{1+4 n^{-1}}}=:\rho_n$. 
Now writing $\pi=r_n^{-2} \int_{D_{r_n}}  \rd s \, \rd t$, from the above we obtain
\begin{align*}
& \textstyle\left|  r_n^{-2} \int_{D_{r_n}} \chi\bigl(\tfrac{\|v-p\|}{r_n }\bigr) v^* \omega\; -\;  \pi  \right| \\
&= \textstyle  r_n ^{-2} \left| \int_{D_{r_n}} \bigl( \chi\bigl(\tfrac{\|v-p\|}{r_n }\bigr)  \omega(\partial_s v , \partial_t v ) \; -\;   1 \bigr) \,  \rd s \, \rd t  \right| \allowdisplaybreaks\\
&\leq  \textstyle  r_n ^{-2} \left( \int_{D_{\rho_n}} \bigl| \omega(\partial_s v , \partial_t v ) - 1 \bigr|   \;+\; \int_{D_{r_n}\setminus D_{\rho_n}} \bigl| \chi\bigl(\tfrac{\|v-p\|}{r_n }\bigr)  \omega(\partial_s v , \partial_t v ) -  1 \bigr|  \right) \\
& \leq r_n ^{-2}  {\textstyle \left( \int_{D_{\rho_n}} \tfrac 4n \;\rd s \, \rd t  
\;+\; \int_{D_{r_n}\setminus D_{\rho_n}} 1\, \rd s \, \rd t  \right) } \\
& = r_n^{-2} \left(  \tfrac 4n \pi \rho_n^2 + 
 \pi r_n^2 -  \pi \rho_n^2 \right)
\;=\; \tfrac{4 \pi n^{-1} }{1- 4 n^{-1}} + \pi - \tfrac{\pi}{1- 4 n^{-1}} 
\quad\underset{n\to\infty}{\longrightarrow}\quad 0. 
\end{align*}
Together with \eqref{eq:A(R)bound} this finishes the proof,
$$\textstyle
A(R) \;\geq\; R^2 \,\lim_{n\to\infty} a(r_n) \;\geq\; R^2 \, r_n^{-2} \int_{D_{r_n}} \chi\bigl(\tfrac{\|v-p\|}{r_n }\bigr) v^* \omega \;\underset{n\to\infty}{\longrightarrow}\; R^2\pi.
$$
\end{proof}
\end{appendix}

\bibliographystyle{alpha}
\bibliography{bibliography.bib}

\end{document}